\documentclass{amsart}

\usepackage{amsaddr}

\usepackage{mathtools}      
\usepackage{graphicx}
\usepackage{color}
\usepackage{epstopdf}
\usepackage{mathptmx}      

\usepackage{constants}
\usepackage{verbatim}
\usepackage{colonequals}
\usepackage{subfigure}

\def\RR{{\mathbb{R}}}
\def\NN{{\mathbb{N}}}

\DeclareMathOperator{\dist}{dist}
\DeclareMathOperator{\tr}{tr}

\DeclareMathOperator*{\argmin}{arg\,min}

\DeclareMathOperator{\Rm}{Rm}
\DeclareMathOperator{\diam}{diam}
\DeclareMathOperator{\inj}{inj}

\newcommand{\Energy}{\mathfrak{J}}
\newcommand{\multind}[1]{\vec{#1}}
\newcommand{\polyp}{\hat{v}}
\newcommand{\gfe}{\mathcal{S}_{h}^{m}}

\theoremstyle{plain}
\newtheorem{theorem}{Theorem}[section]
\newtheorem{lemma}[theorem]{Lemma}
\newtheorem{proposition}[theorem]{Proposition}
\newtheorem{corollary}[theorem]{Corollary}

\theoremstyle{definition}
\newtheorem{definition}[theorem]{Definition}
\newtheorem*{example}{Example}
\newtheorem{condition}{Condition}[section]

\theoremstyle{remark}
\newtheorem{remark}[theorem]{Remark}

\let\oldtocsection=\tocsection

\let\oldtocsubsection=\tocsubsection

\let\oldtocsubsubsection=\tocsubsubsection

\renewcommand{\tocsection}[2]{\hspace{0em}\oldtocsection{#1}{#2}}
\renewcommand{\tocsubsection}[2]{\hspace{1em}\oldtocsubsection{#1}{#2}}
\renewcommand{\tocsubsubsection}[2]{\hspace{2em}\oldtocsubsubsection{#1}{#2}}

\makeatletter
\providecommand\@dotsep{5}
\makeatother

\begin{document}

\title{$L^{2}$-Discretization Error Bounds for Maps into Riemannian Manifolds}

\author[Hardering]{Hanne Hardering}
\address{Hanne Hardering\\
Technische Universit\"at Dresden\\
Institut f\"ur Numerische Mathematik\\
D-01062 Dresden\\
Germany}
\email{hanne.hardering@tu-dresden.de}
\subjclass[2010]{65N15, 65N30}

\begin{abstract}
We study the approximation of functions that map a Euclidean domain $\Omega\subset \RR^{d}$ into an $n$-dimensional Riemannian manifold $(M,g)$ minimizing an elliptic, semilinear energy in a function set $H\subset W^{1,2}(\Omega,M)$.
The approximation is given by a restriction of the energy minimization problem to 
a family of conforming finite-dimensional approximations $S_{h}\subset H$.
We provide a set of conditions on $S_{h}$ such that we can prove a priori $W^{1,2}$- and $L^{2}$-approximation error estimates comparable to standard Euclidean finite elements.
This is done in an intrinsic framework, independently of embeddings of the manifold or the choice of coordinates.
A special construction of approximations ---geodesic finite elements--- is shown to fulfill the conditions, and in the process extended to maps into the tangential bundle.

\smallskip
\noindent \textbf{\keywordsname.} Geometric finite elements, $L^2$-error bounds, vector field interpolation
\end{abstract}

\maketitle

Energy minimizing maps into and between Riemannian manifolds arise in many contexts, both theoretical and applied.
Existence results for harmonic maps have consequences for curvature and topology \cite{Wood}.
Isoperimetric regions (minimizing the area functional), for example, with large volume center in manifolds asymptotic to Schwarzschild have been explored in the context of general relativity and the ADM mass \cite{eichmair2013}.
In general, the influence of an ambient geometry has been of growing interest in the context of geometric flows like mean curvature flow, Ricci flow and Willmore flow.
More applied examples are the modelling of oriented materials in Cosserat theory \cite{neff2007geometrically}, liquid crystals \cite{alouges1997}, and micromagnetics \cite{Melcher}. Manifold valued harmonic map heat-flow has also been introduced as a regularization in image processing \cite{vese2002numerical}.

The research interest extends beyond energy minimization in ambient Riemannian manifolds.
In ambient spacetimes, spacelike hypersurfaces with vanishing mean curvature are maximizers of the area functional and play a role in general relativity as initial data for solving the Einstein equations.
When considering limits of smooth manifolds or problems in optimal transport, the manifold structure of the ambient space needs to be replaced by the general framework of metric spaces. 
While these prospects are certainly interesting, we will limit our attention to smooth Riemannian target manifolds in the hope that later adaptations can provide a more general theory.
\bigskip

Finite dimensional approximations are a useful tool in the study of continuous problems. In the context of gradient flows a time discretization is often used to prove existence results \cite{Ambrosio}. Space discretization has classically been employed to show interior regularity for solutions of elliptic partial differential equations \cite{Evans}. From an applied point of view the usefulness of having numerically implementable discrete schemes is beyond controversy.

Discrete approximations of maps from non-Euclidean domains, e.g., surfaces or more generally Riemannian manifolds, to a linear space, 
have been studied extensively in \cite{deckelnick2005computation, dziuk2007finite, dziuk2012fully, dziuk2013L2}.
Methods for manifold codomains are often ad hoc constructions specific to a particular energy and manifold \cite{bartels_prohl:2007,bartels2005, wriggers_gruttmann:1993, muench:2007, muench_wagner_neff:2009, muench_neff_wagner:2011, simo_vu-quoc:1986, simo_fox_rifai:1990}.
\cite{bartels_prohl:2007,bartels2005} prove weak convergence of an iteration method yielding a nonconforming finite element approximation of a harmonic map into a sphere, where the constraint is only satisfied in the nodal values.
In general, however, convergence and approximation error estimates for methods used for manifold codomains are rarely addressed in the literature, even for simple energies.\\
Recently, geodesic finite elements have been developed and experimentally studied, providing a method of arbitrary order \cite{sander10,sander12,sander13, SanderNeff}. The definition of geodesic finite elements is completely intrinsic and invariant under isometries of the target manifold.
In \cite{grohsSanderH} a systematic approach to generalize standard Euclidean techniques led to $W^{1,2}$-error estimates for the minimization of $W^{1,2}$-elliptic energies to functions with Euclidean domain and Riemannian manifold codomain, in particular for the discretization by geodesic finite elements.

We believe that both the actual discretization and its theory should be intrinsic, i.e., independent of coordinates or (Nash-) embeddings into ambient flat spaces.
In this paper, we follow the intrinsic approach. This conforms to the general paradigm of discretizations obeying the symmetries of their continuous counterparts (in this case, the diffeomorphism group), and appears to be necessary for future extensions of the theory to nonsmooth metric spaces.
\bigskip

In this work, we move further towards an intrinsic discretization error theory for energy minimization problems of the form
\begin{align*}
u:\Omega\to M,\qquad u=\argmin_{v\in H} \Energy (v),
\end{align*}
where $\Omega\subset\RR^{d}$, $M$ is a smooth Riemannian manifold, and $H\subset W^{1,2}(\Omega,M)$.
We restrict our focus to energy functionals $\Energy: H\to \RR$ that are elliptic and semilinear.
Ellipticity is here defined as pathwise convexity and quadratic boundedness with respect to an intrinsic $W^{1,2}$-error measure, that is locally equivalent to the $W^{1,2}$-norm in an embedding. Semilinearity is formulated as a bound on the third variation of $\Energy$. Both properties are shown for the harmonic energy under suitable curvature restrictions on the manifold $M$.

We consider a family of conforming finite-dimensional approximations $S_{h}^{m}\subset H$, that fulfill generalizations of the usual approximability conditions and inverse estimates \cite{ciarlet}.
The parameter $h$ comes from a tesselation of the domain $\Omega$ with elements of diameter proportional to $h$,  while $m$ corresponds to an approximation order.
By restricting $\Energy$ to $S_{h}^{m}$, we obtain abstract discrete problems
\begin{align*}
u_{h}=\argmin_{v_{h}\in S_{h}^{m}} \Energy(v_{h}).
\end{align*}
The main result of this article is an a priori estimate of the form
\begin{align*}
d_{L^{2}}(u,u_{h})\leq C\;h^{m+1}.
\end{align*}
The constant depends on the smoothness of the manifold $M$ as well as the exact solution $u$. In particular, our theory is applicable only in the regular case, i.e., if exact solutions are sufficiently smooth.
This seriously limits the applicability of the theory, as in higher dimensions even the model problem for harmonic maps may have an everywhere discontiuous solution \cite{Riviere}.
Nevertheless, it is a step towards an intrinsic approximation error theory and complements the analogous $W^{1,2}$-estimate from~\cite{grohsSanderH}
\begin{align*}
d_{W^{1,2}}(u,u_{h})\leq C\;h^{m}.
\end{align*}
The proof follows closely the one for the linear case by generalizing the Aubin--Nitsche-Trick \cite{ciarlet}. 
In particular, we define a dual problem that lives on the tangent bundle along the solution $u$. The main ingredients for the estimate are a generalized form of Galerkin orthogonality, $H^{2}$-regularity of the dual problem, and the semilinearity of the energy. The major difference to the linear setting is that all estimates need to be preserved during transport between the different pullbacks of the tangent bundle along $u$ and $u_{h}$, which is ensured by a generalization of a uniformity estimate in~\cite{grohsSanderH}, and appropriate bounds on the continuous and discrete solutions.

We will show that geodesic finite elements fulfill the approximation qualities needed to apply the theory.
In order to do this, we also extend the method including interpolation error estimates to the approximation of vector fields by choice of a particular (pseudo)-metric on the tangent bundle.
Numerical experiments confirming the a priori bounds for geodesic finite elements have been published in \cite{sander12, sander13}.

\tableofcontents
%%%%%%%%%
\section{Sobolev Spaces with Riemannian Manifold Targets}\label{sec:ch1}
%%%%%%%%%%%%%%%%%%%% Manifold valued Sobolev maps %%%%%%%%%%%%%%%%%%%%%%%%%%
%
%
\subsection{Manifold-Valued Sobolev Maps}
In the context of the approximation of variational problems,
the concept of weak solutions and Sobolev spaces arises naturally. In the course of this work we will discuss energy functionals whose arguments are mappings from some open subset $\Omega\in \RR^{d}$ with piecewise Lip\-schitz boundary $\partial\Omega$ into a smooth Riemannian manifold $(M,g)$ without boundary.
\subsubsection{Manifold-Valued $L^{p}$-Maps}
The definition of $L^{p}$-functions can be generalized in a straightforward fashion
to maps taking their values in a complete length space $(X,\dist)$ (see, e.g., \cite{Jost11}).

\begin{definition}\label{def:Lp}
	Let $(X,d)$ be a complete length space and $1\leq p<\infty$.  We define
	\begin{multline*}
	L^{p}(\Omega, X)\colonequals
	\left\{u:\Omega\to M\;|\; u\; \textrm{measurable},\; u(\Omega)\;\textrm{separable}\;, \vphantom{\int_{\Omega}}\right.\\
	\left.
	\int_{\Omega}d^{p}(u(x),Q)\;dx <\infty\; \textrm{for some } Q\in X \right\}.
	\end{multline*}
	We further define a distance map $d_{L^{p}}$ on $L^{p}(\Omega,X)$ by
	\begin{align*}
	d^{p}_{L^{p}}(u,v)\colonequals \int_{\Omega}d^{p}(u(x),v(x))\;dx.
	\end{align*}
\end{definition}

\begin{remark}
	It is easy to see that $d_{L^{p}}$ is indeed well defined for $u,v\in L^{p}(\Omega,M)$, 
	as $x\mapsto (u(x),v(x))$ is a measurable map to $M\times M$.
	Further, the definition of $L^{p}(\Omega, M)$ is independent of the point $Q\in M$ in the definition.
\end{remark}

\subsubsection{Manifold-Valued Sobolev Maps}
In this work we focus on the case where the codomain is a complete Riemannian manifold $(M,g)$. As this is a special case of a complete length space, the definition of $L^{p}(\Omega,M)$ is immediate. This is not true for the Sobolev space $W^{k,p}(\Omega,M)$. The most common definition (see, e.g., \cite{Wood, Hajlasz, Helein, struwe1985}) for Sobolev spaces with Riemannian manifold codomains uses the Nash embedding theorem \cite{Nash}.
\begin{definition}\label{def:NashSobolev}
	Let $(M,g)$ be a $n$-dimensional Riemannian manifold of class $C^{k}$. Let $\iota :M\to \RR^{N}$ be an isometric embedding into Euclidean space.
	We then define
	\begin{align*}
	W_{\iota}^{k,p}(\Omega, M) \colonequals \{ v \in W^{k,p}(\Omega, \RR^{N})\;|\; v(x)\in \iota(M) \; \text{a.e.}\}.
	\end{align*}
\end{definition}
\begin{remark}\label{R:indepIota}
	If the closures of $\Omega$
	and $M$ are compact, then $W_{\iota}^{k,p}(\Omega, M)$ is independent of $\iota$ (see, e.g., \cite{Wood}).
\end{remark}
An important question is whether Sobolev maps can be approximated by smooth maps.
For $p>d$, the result follows by the Sobolev embedding theorem.
For $k=1$, $p=d$ the positive answer is due to Schoen and Uhlenbeck \cite{Uhlenbeck}.
If $p<d$, the corresponding result is no longer true in general. For a survey of these and other known results see \cite{Wood}.

%%%%%%%%%%%%%%%%%%%
For smooth maps $u:\Omega \to M$ the Euclidean definition of $W^{k,p}(\Omega,\RR)$ directly transfers to sections in $u^{-1}TM$, 
i.e., to vector fields $V:\Omega \to TM$ such that $V(x)\in T_{u(x)}M$ for almost every $x\in \Omega$ (cf. \cite{Jost11}).
\begin{definition}\label{def:covDerivAlongU}
	Let $(M,g)$ be an $n$-dimensional Riemannian manifold, $u\in C^{\infty}(\Omega,M)$, and $V\in L^{p}(\Omega,u^{-1}TM)$, 
	i.e., $V:\Omega \to TM$ with $V(x)\in T_{u(x)}M$ almost everywhere in $\Omega$, and
	\begin{align*}
	\int_{\Omega}|V(x)|_{g(u(x))}^{p}\; dx <\infty.
	\end{align*}
	Let $\eta\in C^{\infty}(\Omega,u^{-1}TM)$ with compact support.
	For $\alpha=1,\ldots,d$,
	the covariant derivative of $\eta$ along $u$ 
	is a vector field along $u$ defined by
	\begin{align*}
	\nabla_{du^{\alpha}}\eta(x)\colonequals \lim_{h\to 0}\frac{1}{h}\left(\pi_{u(x)}^{t\mapsto u(x+te_{\alpha})}(\eta(x+he_{\alpha}))-\eta(x)\right),
	\end{align*}
	where $e_{\alpha}\in \RR^{d}$ denotes the $\alpha$-th Euclidean unit vector, and $\pi_{u(x)}^{t\mapsto u(x+te_{\alpha})}$ is the parallel transport
	along the curve defined by $t\mapsto u(x+te_{\alpha})$.
	In coordinates we can write
	\begin{align*}
	\left(\nabla_{du}\eta\right)_{\alpha}^{k}\colonequals \frac{\partial}{\partial x^{\alpha}}\eta^{k}(x) + \Gamma_{ij}^{k}(u(x))du^{i}\bigg(\frac{\partial}{\partial x_{\alpha}}\bigg|_{x}\bigg)\eta^{j}(x),
	\end{align*}
	where Greek indices range from $1$ to $d$, Latin indices range from $1$ to $n$, and we sum over repeated indices.\\
	We say that $V$ is in $W^{1,p}(\Omega,u^{-1}TM)$ if the partial derivatives in the 
	definition of $\left(\nabla_{du}V\right)^{k}_{\alpha}$ exist in a weak sense and are in $L^{p}(\Omega,u^{-1}TM)$.
	We denote $\nabla_{\alpha}V\colonequals \left(\nabla_{du}V\right)_{\alpha}$.
\end{definition}
\begin{remark}
	If $u$ is smooth enough, the Sobolev embedding theorem for vector fields $V\in W^{k,p}(\Omega,u^{-1}TM)$ follows from the Euclidean case by considering $f(x)\colonequals |V(x)|$.
\end{remark}
%
%
% %%%%%%%%%%%%%%%
%
\begin{remark}\label{R:defByDeform}
	If $p>d$, every map in $W^{1,p}(\Omega,M)$ can be described as a pointwise small deformation of a smooth map.
	Thus, $W^{1,p}(\Omega,M)$ can be locally modeled as a Banach manifold over the space of $W^{1,p}$-deformations, 
	i.e., $W^{1,p}(\Omega,u^{-1}TM)$.
	More on this construction can be found in \cite{Palais}.
	This characterization is equivalent to Definition~\ref{def:NashSobolev}.
	Indeed, $u\in W_{\iota}^{1,p}(\Omega,M)$ with $p>d$ implies that $u$ is continuous and hence the image of $\Omega$ is contained in a compact ball $B_{R}$ in $M$.
	By Remark~\ref{R:indepIota}, the definition of $W_{\iota}^{1,p}(\Omega,B_{R})$ is independent of $\iota$.
	Note, however, that the radius $R$ depends on $u$.
\end{remark}
\begin{remark}
	Although $W^{1,2}(\Omega,M)$ is only a manifold for $d=1$, we can nevertheless consider the $W^{1,2}$-norm 
	for vector fields along maps in $W^{1,q}(\Omega,M)$ if $q>\max\{2,d\}$, as those maps are continuous and thus local charts can be used to define covariant derivatives.
\end{remark}
\subsubsection*{Traces of Sobolev Maps}
In the following we will mostly restrict the analysis to maps $u:\Omega\to M$ that are continuous. Thus, we do not really need to concern ourselves with traces of Sobolev maps. An overview on those can be found in \cite{Wood}.

Given boundary and homotopy data $\phi:\Omega\to M$, we set $W^{k,p}_{\phi}(\Omega,M)$ to be the set of maps $v\in W^{k,p}(\Omega,M)$ for which $\tr(v)=\tr(\phi)$ and $v$ and $\phi$ are of the same homotopy class, i.e., there exists a continuous homotopy connecting $v$ and $\phi$.
\subsection{Smoothness Descriptors}\label{sec:SD}
In Euclidean space Sobolev maps are characterized by having a finite norm. We want to obtain a similar intrinsic characterization for maps into Riemannian manifolds. In particular, we want to replace the concept of a Sobolev half-norm.
The new characterization, the so-called smoothness descriptor (cf. \cite{grohsSanderH}), will be intrinsic as well as equivalent to Sobolev half-norms in embeddings $\iota:M\hookrightarrow \RR^{N}$ under reasonable conditions.
Furthermore, it is subhomogeneous with respect to scaling of the domain $\Omega\in \RR^{d}$.

First we note that we can make sense of weak covariant derivatives by using local charts on the target manifold $(M,g)$ if the map we study is continuous.

For covariant differentiation we cannot use the usual multi-index notation as covariant derivatives do not commute.
In the following we use for multiple covariant derivatives the multi-index notation
\begin{align}
\label{eq:multiple_covariant_derivative}
\nabla^{\multind{\beta}} u
\colonequals \nabla_{d^{\beta_k}u} \dots \nabla_{d^{\beta_2}u} d^{\beta_1}u,
\qquad \multind{\beta}\in \{1,\dots,d\}^k,\ k\in \NN_0,
\end{align}
where $d^{\beta}u=du(\frac{\partial}{\partial x_{\beta}})$, 
and $\nabla_{du}$ denotes the covariant derivative along $u$ as defined in Definition~\ref{def:covDerivAlongU}.

For a shorter notation we denote the length of $\multind{\beta}$ by $|\multind{\beta}|$, and set
\begin{align*}
[d] \colonequals \{1,\dots , d\},
\end{align*}
$\nabla^{\multind{\beta}}u\colonequals 1\in \RR$ if $|\multind{\beta}|=0$, and
for $k\geq 1$, 
\begin{align*}
\int_{\Omega}|\nabla^{k}u|^{p}\;dx
&=\sum_{|\multind{\beta}|=k} \int_\Omega \big|\nabla^{\multind{\beta}}u(x)\big|_{g(u(x))}^p\;dx.
\end{align*}
This last term is an obvious candidate for a Sobolev half-norm.
However, the standard Sobolev half-norm $|\iota\circ u|_{W^{k,p}(\Omega,\RR^{N})}$ in an embedding includes lower order terms. Application of the chain rule motivates the following definition, first introduced in \cite{grohsSanderH}, and discussed in depth in \cite{diss}.
\begin{definition}[Smoothness Descriptor]\label{def:SD}
	Let $k\geq 1$, $p\in [1,\infty]$.
	The homogeneous $k$-th order smoothness descriptor of a map $u\in W^{k,p} \cap C(\Omega,M)$ is defined by
	\begin{align*}
	\dot{\theta}_{k,p,\Omega}(u)\colonequals 
	\Bigg(\sum_{\genfrac{}{}{0pt}{}{\multind{\beta}_j \in [d]^{m_j},\ j=1,\dots , l}{\sum_{j=1}^l m_j = k} }
	\int_\Omega \prod_{j=1}^l \left|\nabla^{{\multind{\beta}}_j}u(x)\right|_{g(u(x))}^p\;dx\Bigg)^{1/p},
	\end{align*}
	with the usual modifications for $p=\infty$.
	For $k=0$, and a fixed reference point $Q\in M$, we set
	\begin{align*}
	\dot{\theta}_{0,p,\Omega;Q}(u)\colonequals \left(\int_{\Omega}d^{p}(u(x),Q)\;dx\right)^{1/p}.
	\end{align*}
	Further, we set
	\begin{align*}
	\dot{\theta}_{0,p,\Omega}(u)\colonequals\min_{Q\in M} \dot{\theta}_{0,p,\Omega;Q}(u).
	\end{align*}
	The corresponding inhomogeneous smoothness descriptor is defined by
	\begin{align*}
	\theta_{k,p,\Omega}(u)\colonequals \left(\sum_{i=0}^k \dot\theta^{p}_{i,p,\Omega}(u)\right)^{\frac{1}{p}}.
	\end{align*}
\end{definition}
As the map $u$ in Definition~\ref{def:SD} is continuous, the condition $u\in W^{k,p}(\Omega, M)$ can be interpreted in terms of weak derivatives in local charts (cf. \cite{Jost11}).

Technically, the first order and $k$-th order terms are enough to characterize the smoothness descriptor.
Indeed, these terms bound all other terms of the smoothness descriptor as well as terms that have the same structure but do not directly appear in the smoothness descriptor.
\begin{proposition}[{\cite{diss}}]\label{prop:SDExtraTerms}
	Let $u\in W^{k,p}\cap C(\Omega, M)$, and let $\multind{\alpha}$ be a multi-index in the sense of \eqref{eq:multiple_covariant_derivative} with $|\multind{\alpha}|=l+1$, $0\leq l\leq k-1$.
	Then 
	\begin{align*}
	\left(\int_{\Omega} |\nabla^{\multind{\alpha}} u|^{\frac{kp}{l+1}}\;dx\right)^{\frac{1}{p}}
	&\leq C\;\left(\int_{\Omega}|\nabla^{k}u|^{p}\;dx +  \int_{\Omega}|du|^{kp}\;dx\right)^{\frac{1}{p}}.
	\end{align*}
	In particular this implies
	\begin{align*}
	\left(\int_{\Omega}|\nabla^{k}u|^{p}\;dx +  \int_{\Omega}|du|^{kp}\;dx\right)^{\frac{1}{p}}\leq \dot{\theta}_{k,p,\Omega}(u)\leq C\left(\int_{\Omega}|\nabla^{k}u|^{p}\;dx +  \int_{\Omega}|du|^{kp}\;dx\right)^{\frac{1}{p}}.
	\end{align*}
\end{proposition}
The proof follows from the vector-valued Gagliardo--Nirenberg interpolation inequality \cite{schmeisser2005vector}.

In the Euclidean setting $M=\RR^{n}$ the smoothness descriptor does not 
coincide with the Sobolev norm, as already shown in \cite{grohsSanderH}.
Instead, the two relate in the following way.
\begin{proposition}\label{prop:SDRRn}
	Let $u\in W^{k,p}(\Omega,\RR^{n})$, $k\geq 1$.
	Then
	\begin{align*}
	|u|_{k,p,\Omega}\leq \dot{\theta}_{k,p,\Omega}(u)\leq  C \left(|u|_{k,p,\Omega} + \|du\|^{k}_{0,kp,\Omega} \right)\leq C\|u\|_{k,p,\Omega}^{k}.
	\end{align*}
\end{proposition}
The proof uses Proposition~\ref{prop:SDExtraTerms} and the Sobolev embedding theorem.
Note that we do not need continuity of $u$ in the Euclidean setting.

We can also compare the smoothness descriptor of a map $u\in C(\Omega,M) \cap W_{\iota}^{k,p}(\Omega,M)$
to the smoothness descriptor of the embedded map $\iota\circ u \in W^{k,p}(\Omega,\RR^{N})$.
\begin{proposition}[{\cite{diss}}]\label{prop:SDiota}
	Let $(M,g)$ be compact and of class $C^{k}$, and $\iota:M\to \RR^{N}$ an isometric embedding of class $C^{k}$ such that $0\in \iota(M)$.
	Then for $k\geq 1$ there exist constants depending on $\|\iota\|_{C^{k}}$ such that
	\begin{align}\label{eq:SDiota}
	C\; \dot{\theta}_{k,p,\Omega}(\iota\circ u)\leq \dot{\theta}_{k,p,\Omega}(u) \leq C\; \dot{\theta}_{k,p,\Omega}(\iota\circ u)
	\end{align}
	holds for all $u\in W^{k,p}\cap C(\Omega,M)$.
	For $kp>d$ we have
	\begin{align}\label{eq:spaceEquivSDiota}
	W_{\iota}^{k,p}(\Omega,M)=\left\{v\in C(\Omega,M)\;:\; \theta_{k,p,\Omega}(v) \ \textrm{is well-defined and}\ <\infty \right\},
	\end{align}
	which is independent of $\iota$.
\end{proposition}
\begin{proof}
	Estimate \eqref{eq:SDiota} follows from the chain rule and Proposition~\ref{prop:SDExtraTerms}. Equality \eqref{eq:spaceEquivSDiota} additionally uses the compactness of $M$ and Proposition~\ref{prop:SDRRn}.
\end{proof}
Note that local charts and finiteness of the smoothness descriptor can only be used to characterize $W^{k,p}(\Omega,M)$ if $kp>d$, as these concepts are restricted to continuous functions. This is the reason why Definition~\ref{def:NashSobolev} uses the Nash Embedding Theorem.
\subsubsection{Scaling}
As already observed in \cite{grohsSanderH}, the homogeneous smoothness descriptor is subhomogeneous with respect to rescaling of the domain $\Omega\in \RR^{d}$ with a parameter $h$.
\begin{definition}\label{def:scaling}
	Let $T,T_{h}$ be two domains in $\RR^{d}$, and $F:T_{h}\to T$ a $C^{\infty}$-diffeomorphism.
	For $l\in \NN_{0}$ we say that $F$ scales with $h$ of order $l$ if we have
	\begin{align*}
	\sup_{x\in T}\left|\partial^{\multind{\beta}}F^{-1}(x)\right|\leq C\;h^{k}&\qquad \textrm{for all}\ \multind{\beta}\in [d]^{k},\ k=0,\ldots,l,\\
	\left|\det(DF(x))\right|\sim h^{-d} &\qquad \textrm{for all}\ x\in T_{h}\ (\textrm{where}\ DF\ \textrm{is the Jacobian of}\ F),\\
	\sup_{x\in T_{h}}\left|\frac{\partial}{\partial x^{\alpha}}F(x)\right|\leq C\;h^{-1} &\qquad \textrm{for all}\ \alpha=1,\ldots,d.
	\end{align*}
	Note that as derivatives commute, the multi-indices $\multind{\beta}$ defined by \eqref{eq:multiple_covariant_derivative} can be equivalently replaced by ordinary multi-indices for $\RR^{d}$.
\end{definition}
Readers familiar with finite element theory will recognize such maps $F$ as transformations of an element of a discretization of $\Omega$ to a reference element.
The scale parameter $h$ can also be visualized as the ratio of the diameters.
The smoothness descriptor scales in the following manner.
\begin{lemma}\label{L:scalingSD}
	Let $T,T_{h}$ be two domains in $\RR^{d}$, and $F:T_{h}\to T$ a map that scales with $h$ of order $l$.
	Consider $u\in W^{k,p}(T_{h},M)$ with $1\leq k\leq l$ and $p\in [1,\infty]$. Then
	\begin{align*}
	\dot{\theta}_{k,p,T}(u\circ F^{-1})&\leq C\; h^{k-\frac{d}{p}} \left(\sum_{l=1}^{k} \dot{\theta}^{p}_{l,p,T_{h}}(u)\right)^{\frac{1}{p}}\\
	&\leq C\; h^{k-\frac{d}{p}} \theta_{k,p,T_{h}}(u).
	\end{align*}
\end{lemma}
The proof follows from the chain rule and the integral transformation formula.

\begin{remark}
	Note that Lemmas~\ref{L:scalingSD} only states \emph{sub}homogeneity of the smoothness descriptor, as the homogeneous descriptor is bounded by the inhomogeneous one.
\end{remark}

The third assumption of Definition~\ref{def:scaling} is not needed for the proof of Lemmas~\ref{L:scalingSD}.
It is needed for the following `inverse' estimate.
\begin{lemma}\label{L:inverseScale1}
	Let $T,T_{h}$ be two domains in $\RR^{d}$, and $F:T_{h}\to T$ a map that scales with $h$ of order $1$.
	Consider $u\in W^{1,p}\cap C(T_{h},M)$ with $p\in [1,\infty]$.
	Then
	\begin{align*}
	\dot{\theta}_{1,p,T_{h}}(u)&\leq C\; h^{-1+\frac{d}{p}} \dot{\theta}_{1,p,T}(u\circ F^{-1}).
	\end{align*}
\end{lemma}
\subsubsection{Generalization to Vector Fields}
We now extend the definition of smoothness descriptors to vector fields.
Note that, while vector fields are linear in the sense that $u^{-1}TM$ is a vector space for each $u$, 
the set of all vector fields for all base maps $u$ is not linear.
The idea of the definition is to take essentially a full Sobolev norm of the linear vector field part $V:\Omega\to u^{-1}TM$ but weighting it with covariant derivatives of $u$ to obtain the correct scaling.
\begin{definition}\label{def:nonlinSmoothVfield}
	Let $u \in W^{k,b}\cap C(\Omega,M)$, and $V\in W^{k,p}(\Omega,u^{-1}TM)$, 
	where 
	\begin{align*}
	b\colonequals
	\left\{
	\begin{array}{ll}
	p &\quad\textrm{for}\ kp>d,\\
	p+1 &\quad\textrm{for}\ kp=d,\\
	\frac{d}{k} &\quad\textrm{for}\ kp<d.
	\end{array}
	\right. 
	\end{align*}
	We define the $k$-th order homogeneous smoothness descriptor for vector fields by
	\begin{multline*}
	\dot \Theta_{k,p,\Omega}(V)\colonequals 
	\Bigg(\|V\|^{p}_{L^{a}(\Omega,u^{-1}TM)}\dot\theta^{p}_{k,b,\Omega}(u)\\ +
	\sum_{\stackrel{0\leq o\leq k,\; \multind{\beta}_{j}\in [d]^{m_{j}}}{\sum_{j=0}^{o}m_{j}=k}}
	\int_{\Omega}
	\;|\nabla^{\multind{\beta}_{0}}V(x)|^{p}_{g(u(x))}
	\prod_{j=1}^{o} |\nabla^{\multind{\beta}_{j}}u(x)|^{p}_{g(u(x))}
	\;dx\Bigg)^{1/p},
	\end{multline*}
	where 
	\begin{align*}
	\frac{1}{a}=\frac{1}{p}-\frac{1}{b}.
	\end{align*}
\end{definition}
If $u$ maps $\Omega$ to a constant point $P$ on $M$, then the smoothness descriptor of a vector field $V: \Omega\to T_{P}M$ coincides with the Sobolev half-norm.
For a fixed base function $u$, the smoothness descriptor acts like a half-norm on functions into the linear space $u^{-1}TM$.

As the smoothness descriptor for functions, the smoothness descriptor for vector fields is subhomogeneous with respect to scaling of the domain.
\begin{lemma}\label{L:scalingSDVec}
	Let $T,T_{h}$ be two domains in $\RR^{d}$, and $F:T_{h}\to T$ a map that scales with $h$ of order $l$.
	Consider $u\in W^{k,p}(T_{h},M)$ with $1\leq k\leq l$ and $p\in [1,\infty]$, and $V\in W^{k,p}(T_{h},u^{-1}TM)$. Then
	\begin{align*}
	\dot{\Theta}_{k,p,T}(V\circ F^{-1})&\leq C\; h^{k-\frac{d}{p}} \left(\sum_{l=1}^{k} \dot{\Theta}^{p}_{l,p,T_{h}}(V)\right)^{\frac{1}{p}}\\
	&\leq C\; h^{k-\frac{d}{p}} \Theta_{k,p,T_{h}}(V).
	\end{align*}
\end{lemma}
\subsection{Distances}
Central to this work are errors and thus distances between a minimizer $u$ of $\Energy$ and some finite-dimensional approximation $u_{h}$.
Closely related to the concept of distance is the concept of geodesic.
We can compare geodesics on $M$ with geodesics in $L^{p}(\Omega,M)$ using the concept of geodesic homotopy.
\begin{definition}
	Let $u,v\in L^{p}(\Omega,M)$. We call a map $\Gamma:\Omega\times I\to M$ a 
	\textit{geodesic homotopy} connecting $u$ to $v$ if for almost every $x\in \Omega$ the 
	track curve $\gamma_{x}$
	defined by $\gamma_{x}(t)\colonequals \Gamma(x,t)$ is a constant-speed geodesic connecting $u(x)$ to $v(x)$.
\end{definition}
One can easily verify that geodesic homotopies are $L^{p}$-geodesics.
\subsubsection{The Exponential Map}
The difference between two (close enough) points on a manifold $M$ is characterized by the vector $(p,\log_{p}q)\in TM$,
where $\log_{p}: B_{\textrm{inj}(p)}\to T_{p}M$ denotes the inverse of the exponential map $\exp_{p}:T_{p}M\to M$, and
$\textrm{inj}(p)$ is the injectivity radius at $p\in M$.

The differential of the exponential map is defined by
\begin{align*}
d\exp_{p}V\;:\;T_{p}M\to T_{\exp_{p}V}M,\qquad d\exp_{p}V(W)=\frac{d}{dt}\bigg|_{t=0}\exp_{p}(V+tW).
\end{align*}
For the differential with respect to the base point of $\exp$ we write $d_{2}\exp$, i.e., for $V,W\in T_{p}M$
\begin{align*}
d_{2}\exp_{p} V (W)=\frac{d}{dt}\bigg|_{t=0} \exp_{\gamma_{W}(t)}\pi^{\gamma_{W}}_{\gamma_{W}(0)\mapsto \gamma_{W}(t)}V,
\end{align*}
where $\pi^{\gamma_{W}}$ denotes the parallel transport along $\gamma_{W}$.

For the bivariate logarithm $\log:(p,q)\mapsto \log_{q}p$ we denote the 
covariant derivative with respect to the first and second components by $d$ 
and $d_{2}$, respectively.\\
The following estimates of the derivatives of the logarithm can be proved by direct calculation and Jacobi field estimates. Details can be found in~\cite{diss}.
The bound \eqref{eq:logest1} can also be found in \cite{Karcher1977}. 
\begin{proposition}\label{prop:logest}
	Let $p,q\in B_{\rho}\subset M$ with $\rho$ small enough. Let $\Rm$ denote the Riemannian curvature 
	tensor of $M$, and assume $\Rm$ and $\nabla \Rm$ to be bounded.
	Then
	\begin{align}
	\|d_{2}\log_{p}q+Id\| + \|d\log_{p}q - \pi_{q\mapsto p}\|&\leq |\Rm|_{\infty}\;d^{2}(p,q) \label{eq:logest1},
	\end{align}
	where $\pi_{q\mapsto p}:T_{q}M\to T_{p}M$ denotes parallel transport along a geodesic.
	For a third point $r\in B_{\rho}$, we have
	\begin{align}
	\left|\log_{p}q - \log_{p}r + d\log_{p}q(\log_{q}r)\right|
	&\leq \frac{1}{2}|\Rm|_{\infty} d(p,r)d(p,q)(d(p,q)+d(p,r)).
	\end{align}
	Finally,
	\begin{align}
	\|d_{2}d\log_{p}q\| + \|d^{2}_{2}\log_{p}q\|&\leq C\;d(p,q) \label{eq:logest2},
	\end{align}
	where the constant depends on $\Rm$ and $\nabla \Rm$.
\end{proposition}
\subsubsection{Sobolev Distance as a Metric on $L^{p}(\Omega,TM)$}\label{sec:metricsTM}
Definition~\ref{def:NashSobolev} implies a notion of Sobolev distance based on the embedding $\iota:M\to \RR^{N}$
\begin{align}
d_{W^{1,p}_{\iota}(\Omega,M)}(u,v)\colonequals \|\iota\circ u-\iota\circ v\|_{W^{1,p}(\Omega,\RR^{N})}.
\end{align}
Geodesics for this distance depend on the embedding.
Our goal is to introduce an equivalent concept that is intrinsic.
In particular we want to obtain one class of distance-realizing curves for all $W^{k,p}$-distances, independent of $k$.
(Note that we only consider $k=0,1$ here, but the ideas generalize to arbitrary $k$.)

The difference of two maps $u,v\in C(\Omega,M)$ is characterized by the pointwise difference $(u(x),\log_{u(x)}v(x))\in TM$ if $u$ and $v$ are close enough to each other.
In order to characterize the difference between the differentials of two maps $u,v\in C^{1}(\Omega,M)$, 
i.e., $(u(x),d^{\alpha}u(x)), (v(x),d^{\alpha}v(x)) \in TM$ at a point $x\in \Omega$, we consider the tangent bundle itself as a manifold.

There are several natural metrics on the tangent bundle.
A complete classification has been provided in \cite{Kowalski}.
There are in particular two classical constructions, namely the Sasaki metric, 
which is a Riemannian metric  on $TM$, and the horizontal (or complete) lift, 
which is only pseudo-Riemannian.
Geodesics of the Sasaki metric are in general complicated objects, whose projections 
onto the manifold are in general not geodesics in $M$.
This property of geodesics in the tangent bundle projecting to $M$-geodesics is desirable 
for a host of reasons, among them a natural splitting of distances into a part on $M$ and a vector part.\\
In local coordinates the horizontal lift $g^{h}$ is given by
\begin{align*}
g^{h}_{(p,V)}=\left(
\begin{array}{ll}
V^{a}\Gamma_{ai}^{k}g_{kj} + V^{a}\Gamma_{aj}^{k}g_{ki}\quad & 
g_{ij}\\
g_{ij} & 0
\end{array}
\right).
\end{align*}
Geodesics of $g^{h}$ correspond to Jacobi fields along geodesics in $M$ \cite{Casciaro}.
The inverse of the exponential map on $TM$ is defined by
\begin{align*}
^{h}\log_{(p,V_{p})}(q,V_{q})=\left(\log_{p}q, d\log_{p}q (V_{q}) +\;d_{2}\log_{p}q (V_{p}) 
\right)
\end{align*}
for $(p,V_{p}),(q,V_{q})\in TM$ with $d(p,q)\leq \textrm{inj}_{M}(p)$.
The horizontal lift arises naturally when we consider the change of the distance 
between two curves $\gamma$ and $\mu$ in $M$ as
\begin{align*}
\frac{d}{dt}\bigg|_{t=0}\left|\log_{\gamma(t)}\mu(t)\right|_{g}^{2}=\left|\;^{h}\log_{\left(\gamma(0),\dot{\gamma}(0)\right)}\left(\mu(0),\dot{\mu}(0)\right)\right|_{^{h}g}^{2}.
\end{align*}
As $g^{h}$ is only a pseudo-Riemannian metric, it is not meaningful to consider lengths with respect to $g^{h}$. 
We can however consider the Sasaki-length of $g^{h}$-geodesics and---as an approximation---of the $g^{h}$-logarithm.
The length of $g^{h}$-geodesics is of particular interest, 
as the derivatives of geodesic homotopies are $g^{h}$-geodesic homotopies in the following sense.
\begin{lemma}[\cite{diss}]
	Let $u,v\in C^{1}(\Omega,M)$ such that $d_{L^{\infty}}(u,v)\leq \textrm{inj}_{M}(p)$ for all points $p\in u(\Omega)\cup v(\Omega)\subset M$.
	Let $\Gamma:\Omega\times I\to M$ be the geodesic homotopy connecting $u$ to $v$. Then $\Gamma(\cdot,t) \in C^{1}(\Omega,M)$, and for any $x\in \Omega$ and $\alpha\in \left\{1,\ldots,d\right\}$ the curve
	$d^{\alpha}\Gamma(x,\cdot):I\to TM$ describes a $g^{h}$-geodesic in $TM$ connecting $d^{\alpha}u(x)$ to $d^{\alpha}v(x)$.
\end{lemma}
\begin{remark}
	For manifolds with bounded curvature, geodesic homotopies also inherit the weak differentiability of their endpoint maps.
	For first order derivatives this follows by the Rauch comparison principle.
	Indeed, if $\exp$ and $\log$ are in $C^{k}$ in their arguments and $M$ admits a $C^{k}$-embedding into Euclidean space, 
	$u,v\in W^{k,p}\cap C(\Omega,M)$, and $d_{L^{\infty}}(u,v)\leq \textrm{inj}_{M}(p)$ for all $p\in u(\Omega)\cup v(\Omega)\subset M$, 
	then the geodesic homotopy connecting $u$ to $v$ lies in $W^{k,p}\cap C(\Omega,M)$.
	This follows by the chain rule.
	In particular, we can always estimate the homogeneous smoothness descriptor along geodesic homotopy by the inhomogeneous smoothness descriptors at the endpoints
	\begin{align}
	\dot \theta_{k,p,\Omega}(\Gamma(s))\leq C(M)\left(\theta_{k,p,\Omega}(\Gamma(0))+\theta_{k,p,\Omega}(\Gamma(1))\right).
	\end{align}
\end{remark}
In light of these considerations, we can now define the following first-order Sobolev distance measure.
\begin{definition}\label{def:SobolevForC1}
	Let $u,v\in W^{1,p}(\Omega,M)\cap C(\Omega,B_{\textrm{inj}_{M}})$, and $\Gamma$ denote the geodesic homotopy connecting $u$ to $v$.
	We set	
	\begin{align*}
	D^{p}_{1,p}(u,v)
	& \colonequals \sum_{\alpha=1}^{d} \int_{\Omega}\; 
	\|\nabla_{d^{\alpha}u}\log_{u(x)}v(x)\|_{g(u(x))}^{p}\;dx,
	\end{align*}
	and
	\begin{align*}
	d_{W^{1,p}}(u,v)\colonequals d_{L^{p}}(u,v) + D_{1,p}(u,v).
	\end{align*}
\end{definition}
Note that $d_{W^{1,p}}$ defined by Definition~\ref{def:SobolevForC1} is not a distance. Locally, however, we can show equivalence to the Sobolev distance in an embedding. To do so we restrict our considerations to the following ball.
\begin{definition}\label{def:HKL}
	Let $q>\max\{p,d\}$, and let $K$ and $L$ be two constants such that $L\leq \inj(M)$ 
	and $KL\leq \frac{1}{|\Rm|_{\infty}}$.
	We set
	\begin{align}
	W^{1,q}_{K}\colonequals \left\{v\in W^{1,q}(\Omega,M)\;:\; \theta_{1,q,\Omega}(v)\leq K\right\},
	\end{align}
	and denote by $H^{1,p,q}_{K,L}$ an $L$-ball w.r.t.\ $L^{s}$ in $W^{1,q}_{K}$, 
	where 
	\begin{align*}
	s\colonequals
	\left\{
	\begin{array}{ll}
	\frac{pq}{q-p}&\quad\textrm{for}\ d<p\\
	\frac{2pq}{q-p}&\quad\textrm{for}\ d=p\\
	\frac{dq}{q-d}&\quad\textrm{for}\ d>p.
	\end{array}
	\right.
	\end{align*}
\end{definition}
On the restricted set $H^{1,p,q}_{K,L}$ one can prove a uniformity lemma for parallel vector fields along geodesic homotopies \cite{grohsSanderH, diss}. By a this we mean a map $V\in C(\Omega\times I \to \Gamma^{-1}TM)$, such that for every $x\in \Omega$ the vector field $V(x,\cdot)$ is in $W^{1,1}(I,\Gamma(x,\cdot)^{-1}TM)$ and parallel along the curve $\Gamma(x,\cdot):I\to M$.
\begin{lemma}[Uniformity Lemma]\label{L:uniformity}
	Let $u,v \in H^{1,p,q}_{K,L}$ as defined in Definition~\ref{def:HKL}, and let $\Gamma$ be the geodesic homotopy connecting $u$ to $v$.
	Consider a parallel vector field $V\in\ W^{1,p} \cap C(\Omega\times I,\Gamma^{-1}TM)$ along $\Gamma$.
	Then there exists a constant  $\Cl{c:uniformity}$ depending on the curvature of $M$, the Sobolev constant, and the dimension $d$ of $\Omega$ such that
	\begin{align*}
	\frac{1}{1+\Cr{c:uniformity}t}\|V(\cdot,0)\|_{W^{1,p}(\Omega,u^{-1}TM)} &\leq \|V(\cdot,t)\|_{W^{1,p}(\Omega,\Gamma(\cdot,t)^{-1}TM)}\\
	& \leq (1+\Cr{c:uniformity}t)\|V(\cdot,0)\|_{W^{1,p}(\Omega,u^{-1}TM)}
	\end{align*}
	holds for all $t\in I$.
\end{lemma}
The proof of the lemma follows by differentiating $\|V(\cdot,t)\|_{W^{1,p}(\Omega,\Gamma(\cdot,t)^{-1}TM)}$ with respect to $t$ and using the H\"older and Sobolev inequalities.

The uniformity lemma can be used to show, that $d_{W^{1,p}}$ defines a quasi-inframetric on $H^{1,p,q}_{K,L}$.
\begin{definition}\label{def:infra}
	Let $S$ be a set and $D:S\times S \to \RR$ a positive definite mapping.
	We call $D$ a quasi-inframetric if it
	fulfills a relaxed triangle inequality
	\begin{align}\label{eq:infra}
	D(x,y)\leq C\left(D(x,z) + D(z,y)\right)\qquad \forall x,y,z\in S,
	\end{align}
	and is symmetric up to a constant, i.e.,
	\begin{align}\label{eq:quasi}
	D(x,y)\leq C\;D(y,x) \qquad \forall x,y \in S.
	\end{align}
\end{definition}
In particular, we have the following.
\begin{proposition}[{\cite{diss}}]\label{prop:equivEmbD12}
	On $H^{1,p,q}_{K,L}$ the mapping $d_{W^{1,p}}$ is a quasi-inframetric.
	If $\iota:M\to \RR^{N}$ denotes a smooth isometric embedding, then for all $u,v \in H^{1,p,q}_{K,L}$ there exists a constant depending on the curvature of $M$, $\|\iota\|_{C^{2}}$, and $K$, such that
	\begin{align*}
	\|\iota\circ u -\iota\circ v \|_{W^{1,2}(\Omega,\RR^{N})}\leq C \;d_{W^{1,p}}(u,v).
	\end{align*}
	If additionally $d_{L^{\infty}}(u,v)\leq \inj_{M}$, then equilvalence holds, i.e.,
	\begin{align*}
	d_{W^{1,p}}(u,v)\leq C \|\iota\circ u -\iota\circ v \|_{W^{1,2}(\Omega,\RR^{N})}.
	\end{align*}
\end{proposition}
\subsubsection{Scaling and Compatibility with the Smoothness Descriptor}
The Sobolev (half-)metric $D_{1,p}$ is compatible with the smoothness descriptor in two ways:\\
First, it fulfills locally in $H^{1,2,q}_{K,L}$ the following triangle inequality.
\begin{proposition}\label{prop:SDtriangle2}
	Let $u,v \in H^{1,p,q}_{K,L}$. Then there exists a constant $C$ such that
	\begin{align}
	\dot{\theta}_{1,p,\Omega}(v)\leq \dot{\theta}_{1,p,\Omega}(u) + C\;D_{1,p}(u,v).
	\end{align}
\end{proposition}
The proof follows by Lemma~\ref{L:uniformity}.\\
Secondly, under scaling of the domain the $L^p$- and Sobolev metrics behave as follows.
\begin{lemma}\label{L:inverseScale2}
	Let $T,T_{h}$ be two domains in $\RR^{d}$, and $F:T_{h}\to T$ a map that scales with $h$ of order $1$.
	Consider $u,v\in W^{1,p}\cap C(T_{h},M)$ with $p\in [1,\infty]$.
	Then
	\begin{align*}
	d_{L^{p}}(u,v)&\leq C\;h^{\frac{d}{p}}d_{L^{p}}(u\circ F^{-1},v\circ F^{-1})\\
	D_{1,p}(u,v)&\leq C\;h^{\frac{d}{p}-1}D_{1,p}(u\circ F^{-1},v\circ F^{-1}).
	\end{align*}
\end{lemma}
The proof follows from the chain rule and the integral transformation formula.
%

%%%%%%%%%%
\section{Discretization Error Bounds}\label{sec:ch3}
We consider the minimization of energies $\Energy$ in $H\subset W_{\phi}^{1,q}(\Omega,M)$, $q>\max\{2,d\}$
where $\phi:\bar{\Omega}\to M$ denotes suitable boundary and homotopy data:
\begin{align}\label{eq:Pcont}
u\in H:\qquad \Energy(u)\leq \Energy(v)\qquad \forall v\in H.
\end{align}
To bound the error of discrete approximations to minimizers of $\Energy$, we need the concept of $W^{1,2}$-ellipticity.
\begin{definition}\label{def:ellipticEnergy}
	Let $\Energy:H \to \RR$ be twice continuously differentiable along geodesic homotopies.
	We say that $\Energy$ is
	\begin{enumerate}
		\item \label{def:coercive} 
		$W^{1,2}$-coercive, if
		there exists a constant $\lambda>0$ such that for all $v\in H$ and $V\in W_{0}^{1,2}(\Omega, v^{-1}TM)$ we have
		\begin{align}\label{eq:ellipticbelow}
		\lambda\|V\|^2_{W^{1,2}(\Omega,v^{-1}TM)}\leq \frac{d^2}{ds^{2}}_{|s=0}\Energy(\exp_{v}(sV)).
		\end{align}
		\item \label{def:bounded} 
		$W^{1,2}$-bounded, if there exists a constant $\Lambda>0$ such that for all $v\in H$ and
		for all $V,W \in W_{0}^{1,2}(\Omega, v^{-1}TM)$
		we have
		\begin{align}\label{eq:ellipticabove}
		\left|\frac{d^2}{dr\;ds}_{|(r,s)=(0,0)}\Energy(\exp_{v}(sV+rW))\right|\leq \Lambda\;\|V\|_{W^{1,2}(\Omega,v^{-1}TM)}\|W\|_{W^{1,2}(\Omega,v^{-1}TM)}.
		\end{align}
		\item  $W^{1,2}$-elliptic, if \ref{def:coercive} and \ref{def:bounded} hold.
	\end{enumerate}
\end{definition}
In order to obtain a finite-dimensional approximation of $H$, we assume that we have a conforming grid $G$ on $\Omega$, i.e., a partition into polytopes, such that the closures intersect in a common face. 
\begin{definition}\label{def:widthh}
	We say that a conforming grid $G$ for the domain 
	$\Omega\subset 
	\RR^{d}$ is of width $h$ and order $m$, if for each element $T_{h}$ of $G$ there exists a 
	$C^{\infty}$-diffeomorphism $F_{h}:T_{h}\to T$ to a reference element  $T\subset \RR^{d}$ that scales with $h$ of order $m$.
\end{definition}
Let $S^{m}_{h}\subset H$ be a discrete approximation space for a grid $G$ on $\Omega$ of width $h$ and order $m$.
\begin{remark}
	Note that the assumption that $S_{h}^{m}$ is conforming includes that the boundary data $\phi|_{\partial\Omega}$ can be represented exactly in $S_{h}^{m}$. This part of the assumption may be
	waived and replaced by a standard approximation argument for boundary data \cite{ciarlet, grohsSanderH}.
\end{remark}
Consider the discrete approximation of~\eqref{eq:Pcont}
\begin{align}\label{eq:Pdisc}
u_{h}\in S^{m}_{h}:\qquad \Energy(u_{h})\leq \Energy(v_{h})\qquad \forall v_{h}\in S_{h}^{m}.
\end{align}
Variations of discrete functions provide a notion of discrete test vector fields \cite{diss}. For a detailed construction of these in the context of geodesic finite elements see~\cite{SanderTest}.
\begin{definition}\label{def:discreteVfields}
	Let $u_{h}\in S_{h}^{m}$. We denote by $IV_{0}(\Omega,u_{h}^{-1}TM)$ the set of all vector fields $V_{h}\in W^{1,2}(\Omega,u_{h}^{-1}TM)$ with boundary values $0$, such that there exists a family $v_{h}(t)\in S_{h}^{m}$ with $v_{h}(0)=u_{h}$ and $\frac{d}{dt}v_{h}(0)=V_{h}$.
\end{definition}
In order to control the error between $u$ and $u_{h}$, we need to
formulate approximability conditions on the discrete space $S_{h}^{m}$.
This allows us to obtain discretization error estimates for all approximation spaces that fulfill these conditions.
We will later in Section~\ref{sec:ch2} introduce a specific example for $S_{h}^{m}$.

We want to provide an intrinsic theory, i.e., all estimates should change equivalently under isometries of the manifold and be independent of embeddings into Euclidean space. If for example the $(m+1)$-th order derivatives of functions $u_{h}\in S_{h}^{m}$ vanish (as for polynomials in the Euclidean case), this does not mean that the $(m+1)$-th order derivatives of $\iota\circ u_{h}$ with an embedding $\iota:M\to \RR^{m}$ vanish. Thus, we will phrase the conditions in terms of the smoothness descriptor introduced in Section~\ref{sec:ch1} for this purpose.

The first condition consists of an estimate for the best approximation error in $S_{h}^{m}$ as usually also used in the Euclidean setting \cite{ciarlet}.
\begin{condition}\label{cond:1}
	Let $kp>d$, $m\geq k-1$, and $u\in W^{k,p}(\Omega,M)$ with $u(T_{h})\subset B_{\rho}\subset M$, where $\rho\leq \inj(M)$, for all elements $T_{h}\in G$. For small enough $h$ let there exist a map $u_{I}\in S_{h}^{m}$ and constants $\Cl{c:m+1Deriv}, \Cl{c:cond1} $ with 
	\begin{align}\label{eq:cond1b}
	\dot{\theta}_{l,q,T_{h}}(u_{I})\leq \Cr{c:m+1Deriv}\;\dot{\theta}_{l,q,T_{h}}(u)
	\end{align}
	for all $k-\frac{d}{p}\leq l\leq k$ and $q\leq \frac{pd}{d-p(k-l)}$, that fulfills
	on each element $T_{h}\in G$ the estimate
	\begin{align}\label{eq:cond1a}
	d_{L^{p}}(u,u_{I}) + h\;D_{1,p}(u,u_{I}) \leq
	\Cr{c:cond1}\;h^{k}\;\theta_{k,p,T_{h}}(u).
	\end{align}
\end{condition}
Note that the discrete functions in $S_{h}^{m}\subset H$ are globally only of $W^{1,q}$-smoothness.
In accordance with standard theory we define grid dependent smoothness descriptors.

\begin{definition}
	Let $k\geq 1$, $p\in [1,\infty]$, and $u\in C(\Omega,M)$ with $u_{|T_{h}}\in W^{k,p}(T_{h},M)$ 
	for all elements $T_{h}$ from the grid $G$.
	We set
	\begin{align}
	\dot{\theta}_{k,p,G}(u)\colonequals \left(\sum_{T_{h}\in G} \dot{\theta}_{k,p,T_{h}}^{p}(u)\right)^{\frac{1}{p}}.
	\end{align}
	Analogously, for a function $v\in C(\Omega,M)$ with $v_{|T_{h}}\in W^{k,b}(T_{h},M)$
	for all elements $T_{h}\in G$, $b$ as in Definition~\ref{def:nonlinSmoothVfield}, and a vector field $V$ along $v$ such that	$V\in W^{k,p}(T_{h},v^{-1}TM)$ for all $T_{h}\in G$, we set
	\begin{align}
	\dot{\Theta}_{k,p,G}(V)\colonequals \left(\sum_{T_{h}\in G} \dot{\Theta}_{k,p,T_{h}}^{p}(V)\right)^{\frac{1}{p}}.
	\end{align}
\end{definition}
By summation over all elements, estimates like~\eqref{eq:cond1b} and~\eqref{eq:cond1a} carry over to the global grid-dependent smoothness descriptors.

As we later need to approximate the generalized test functions as well, we also need a best approximation error estimate between non-discrete vector fields and variations of discrete maps.
\begin{condition}\label{cond:2b}
	For any element $T_{h}$ from the grid $G$, let $S_{h}^{m}(T_{h},M)\subset W^{2,b}(T_{h},M)$ for
	$b$ as in Definition~\ref{def:nonlinSmoothVfield} with $p=2$.
	Given any $u_{h}\in S_{h}^{m}(T_{h},M)$ and $V\in W^{2,2}(T_{h},u_{h}^{-1}TM)$, let there exist a family of maps $v_{h}(t)\in S_{h}^{m}(T_{h},M)$ with $v_{h}(0)=u_{h}$ and constants $\Cl{c:cond2b}, \Cl{c:cond2b2}$ such that
	for $V_{I}=\frac{d}{dt}v_{h}(0)$ the estimates
	\begin{align}\label{eq:dgeq4bound}
	\Theta_{2,2,T_{h}}(V_{I})\leq \Cl{c:cond2b}\; \Theta_{2,2,T_{h}}(V)
	\end{align}
	and
	\begin{align}\label{eq:vecInterpolCond}
	\| V - V_{I}\|_{W^{1,2}(T_{h},u_{h}^{-1}TM)} &\leq \Cl{c:cond2b2} h \;
	\Theta_{2,2,T_{h}}(V)
	\end{align}
	hold.
\end{condition}
Conditions~\ref{cond:1} and~\ref{cond:2b} are sufficient to prove approximation errors locally close to the exact solution $u$ of~\eqref{eq:Pcont}. This means that the set of discrete functions $S_{h}^{m}$ in~\eqref{eq:Pdisc} has to be restricted by additional bounds. We still obtain a meaningful local result if we can show that the discrete solution $u_{h}$ stays away from these bounds. In order to do this, we 
need the following condition generally known as an inverse estimate.
\begin{condition}\label{cond:3}
	On a grid $G$ of width $h$ and order $m$, under the additional assumption that $F^{-1}_{h}:T\to T_{h}$ scales with order $2$ for every $T_{h}\in G$, 
	for $p,q\in[1,\infty]$ let there exist constants $\Cl{c:inverseEstref11}$ and $\Cl{c:inverseEstref21}$ such that		
	\begin{align}\label{eq:inverse2p1qScaled}
	\dot{\theta}_{1,p,T_{h}}(v_{h})&\leq \Cr{c:inverseEstref11}\;\;h^{-d\max\left\{0,\frac{1}{q}-\frac{1}{p}\right\}}\;\dot{\theta}_{1,q,T_{h}}(v_{h})\\
	\dot{\theta}_{2,p,T_{h}}(v_{h}) &\leq \Cr{c:inverseEstref21}\;\dot{\theta}^{2}_{1,2p,T_{h}}(v_{h}) + \Cr{c:inverseEstref21}\;h^{-1-d\max\left\{0,\frac{1}{q}-\frac{1}{p}\right\}}\;\dot{\theta}_{1,q,T_{h}}(v_{h})
	\end{align}
	for any $v_{h}\in S^{m}_{h}(T_{h},M)$ with $v(T_{h})\subset B_{\rho}$ for $\rho$ small enough.
\end{condition}
Once we can show that the discrete solution is indeed close to the continuous one, we can infer even stronger bounds on higher derivatives of it from the exact solution. In order to do so, we need inverse estimates on differences of discrete functions, i.e., vector fields of the form $\log_{v_{h}}w_{h}$ with $v_{h},w_{h}\in S^{m}_{h}(T_{h},M)$. Note that these are not discrete vector fields in the sense of Definition~\ref{def:discreteVfields} themselves.
\begin{condition}\label{cond:4}
	On a grid $G$ of width $h$ and order $m$, under the additional assumption that $F^{-1}_{h}:T\to T_{h}$ scales with order $2$ for every $T_{h}\in G$, for $p,q\in[1,\infty]$ let there exist constants $\Cl{c:vecInv0}, \Cl{c:vecInv2}, \Cl{c:vecInvm}$ such that		
	\begin{align}\label{}
	\|\log_{v_{h}}w_{h}\|_{L^{p}(T_{h}, v_{h}^{-1}TM)}&\leq \Cr{c:vecInv0}\;h^{-d\max\left\{0,\frac{1}{q}-\frac{1}{p}\right\}}\|\log_{v_{h}}w_{h}\|_{L^{q}(T_{h}, v_{h}^{-1}TM)}\\
	|\log_{v_{h}}w_{h}|_{W^{2,p}(T_{h}, u_{h}^{-1}TM)}&\leq \Cr{c:vecInv2}\dot\theta_{2,p,T_{h}}(v_{h}) +\Cr{c:vecInv2}\dot\theta_{1,2p,T_{h}}^{2}(w_{h})\\
	&\quad   + \Cr{c:vecInv2} h^{-1-d\max\left\{0,\frac{1}{q}-\frac{1}{p}\right\}}\|\log_{v_{h}}w_{h}\|_{W^{1,q}(T_{h}, v_{h}^{-1}TM)}\notag\\ 
	&\quad+ \Cr{c:vecInvm} h^{-2}\|\log_{v_{h}}w_{h}\|_{L^{p}(T_{h}, v_{h}^{-1}TM)}\notag
	\end{align}
	for any $v_{h},w_{h}\in S^{m}_{h}(T_{h},M)$ with $w(T_{h}),v(T_{h})\subset B_{\rho}$ for $\rho$ small enough.
	For $m=1$, the constant $\Cr{c:vecInvm}$ must be zero.
\end{condition}

\subsection{$W^{1,2}$-Error Bounds}
We recall the $W^{1,2}$-discretization error bounds given in \cite{grohsSanderH}, in particular the generalized C\'ea Lemma:
\begin{lemma}\label{L:cea}
	Assume that $u\in H$ is a minimizer of $\Energy:H\to \RR$ w.r.t.\ variations along geodesic homotopies in $H$, and
	that $\Energy$ is elliptic along geodesic homotopies starting in $u$.\\
	For $K> \theta_{1,q,\Omega}(u)$, $L\leq \inj(M)$, and $KL\leq \frac{1}{|\Rm|_{\infty}}$ let $W^{1,q}_{K}$ and $H^{1,2,q}_{K,L}$
	be defined by Definition~\ref{def:HKL}, and set
	\begin{align*}
	H_{K,L}\colonequals H\cap H^{1,2,q}_{K,L}.
	\end{align*}
	Consider a subset $V_{h} \subset H_{K,L}$ such that	
	\begin{align*}
	w = \argmin_{v\in V_{h}}\Energy(v)
	\end{align*}
	exists.\\
	Then
	\begin{equation}\label{eq:cea}
	D_{1,2}(u,w) \leq (1+\Cr{c:uniformity})^{2}\sqrt{\frac{\Lambda}{\lambda}} \inf_{v\in V}D_{1,2}(u,v)
	\end{equation}
	holds, where $\Cr{c:uniformity}$ is the constant appearing in Lemma~\ref{L:uniformity}.
\end{lemma}
A combination of this version of C\'ea's Lemma with Condition~\ref{cond:1} yields the $W^{1,2}$-error estimate shown in \cite{grohsSanderH}.
\begin{theorem}\label{T:H1err}
	Let $2(m+1)>d$, and $m\geq 1$. Assume that $u\in W_{\phi}^{m+1,2}(\Omega,M)$ is a minimizer of $\Energy:H\to \RR$ w.r.t.\ variations along geodesic homotopies in $H$, and
	that $\Energy$ is elliptic along geodesic homotopies starting in $u$.
	\\
	For a conforming grid $G$ of width $h$ and order $m$ (cf. Definition~\ref{def:widthh}) set $V_{h}\colonequals 
	H\cap S_{h}^{m}$.
	Assume that the boundary data $\phi_{|\partial \Omega}$ is such that $V_{h}$ is not empty.
	\\
	Let $u_{I}\in S_{h}^{m}$ be the approximating map from Condition~\ref{cond:1}.
	Let $K$ be a constant such that
	\begin{align}\label{eq:Kbuffer}
	K\geq \Cr{c:m+1Deriv}\theta_{1,q,\Omega}(u),
	\end{align}
	and assume that $h$ is small enough such that $u_{I}\in H_{K,L}$, where $H_{K,L}$ is defined as in Lemma~\ref{L:cea}.
	\\
	Then the discrete minimizer
	\begin{align*}
	u_{h}\colonequals \argmin_{v_{h}\in V_{h}\cap H_{K,L}} \Energy(v_{h})
	\end{align*}
	fulfills the a priori error estimate
	\begin{align}\label{eq:H1err}
	D_{1,2}(u,u_{h})\leq \Cl{c:H1err} h^{m}\theta_{m+1,2,\Omega}(u).
	\end{align}
	If the error is measured in an isometric embedding $\iota:M\to \RR^{N}$, we have
	\begin{align}\label{eq:H1errEmb}
	\|\iota\circ u - \iota\circ u_{h}\|_{W^{1,2}(\Omega,\RR^{N})}\leq C\;h^{m}\theta_{m+1,2,\Omega}(\iota\circ u)\leq C\;h^{m}\|\iota\circ u\|_{m+1,p,\Omega}^{m+1}.
	\end{align}
\end{theorem}
Note that C\'ea's Lemma and Theorem~\ref{T:H1err} do not need a variational formulation of the problem, but instead use minimization. Thus, they do not conflict with the additional $K$- and $L$-bounds on the discrete functions. We can show that if Condition~\ref{cond:3} holds, and if the grid $G$ fulfills indeed the additional assumption that $F^{-1}_{h}:T\to T_{h}$ scales with order $2$ for every $T_{h}\in G$, then we can choose $q$ such that the restricted solution $u_{h}\in V_{h}\cap H_{K,L}$ stays away from the $K$ and $L$ bounds and is thus indeed a local solution in $V_{h}$. This is the content of the following lemma.
\begin{lemma}[{\cite{diss}}]\label{L:chooseq}
	Assume that Condition~\ref{cond:3} holds, and that the grid $G$ fulfills the additional assumption that $F^{-1}_{h}:T\to T_{h}$ scales with order $2$ for every $T_{h}\in G$.
	Let $u\in W_{\phi}^{m+1,2}(\Omega,M)$, $2(m+1)>d$, $u_{h}\in H\cap S_{h}^{m}$ with $\dot\theta_{1,q,\Omega}(u_{h})\leq K$, and
	\begin{align*}
	D_{1,2}(u,u_{h})\leq C\;h^{m}\theta_{m+1,2,\Omega}(u).
	\end{align*}
	Then we can choose $q>\max\{d,2\}$ such that
	\begin{align}
	D_{1,q}(u,u_{h})\leq C\;h^{\delta}
	\end{align}
	holds with some $\delta>0$ and a constant depending on $K$ and $u$. Note that this also implies
	\begin{align}
	d_{L^{\infty}}(u,u_{h})\leq C\;h^{\delta},
	\end{align}
	as well as $\theta_{1,q,\Omega}(u_{h})<K$, as long as $K>\theta_{1,q,\Omega}(u)$ and $h$ small enough.
\end{lemma}

\begin{proof}
	Let
	\begin{align*}
	r\colonequals \left\{\begin{array}{ll}
	\infty &\ \textrm{for}\ d<2m\\
	\epsilon^{-1} &\ \textrm{for}\ d=2m\\
	\frac{2d}{d-2m} &\ \textrm{for}\ 2m<d<2(m+1),
	\end{array}\right.
	\end{align*}
	with $\epsilon>0$ arbitrary, and choose $\max\{2,d\}<q<r$. We set
	\begin{align*}
	\mu&\colonequals \left(\frac{1}{2}-\frac{1}{q}\right)\left(\frac{1}{q}-\frac{1}{r}\right)^{-1},&
	\delta&\colonequals \frac{m-d\left(\frac{1}{2}-\frac{1}{q}\right)}{1+\mu}.
	\end{align*}
	Then by choice of $q$, we have $\delta>0$ and by $L^{p}$-interpolation
	\begin{align*}
	D_{1,q,\Omega}(u,u_{h})&\leq h^{\delta-m}\;D_{1,2,\Omega}(u,u_{h}) + h^{\frac{m-\delta}{\mu}}\;D_{1,r,\Omega}(u,u_{h})\\
	&\leq C\;h^{\delta}\;\theta_{m+1,2,\Omega}(u) + C\;h^{\frac{m-\delta}{\mu}}\;\left(\dot\theta_{1,r,\Omega}(u) + \dot\theta_{1,r,\Omega}(u_{h})\right)\\
	&\leq  C\;h^{\delta}\;\theta_{m+1,2,\Omega}(u) + C\;h^{\frac{m-\delta}{\mu} - d\left(\frac{1}{q}-\frac{1}{r}\right)}\; \dot\theta_{1,q,\Omega}(u_{h})\\
	&= C\;h^{\delta}\left(\theta_{m+1,2,\Omega}(u) + \dot\theta_{1,q,\Omega}(u_{h}) \right).\qedhere
	\end{align*}
\end{proof}

\subsection{The Aubin--Nitsche Lemma and Predominantly Quadratic Energies}
Our goal is to show that for $W^{1,2}$-elliptic minimization problems the $L^{2}$-discretization error is in $O(h^{m+1})$, where $m$ is the approximation order. 

We recall the Aubin--Nitsche lemma for the approximation of a quadratic minimization problem in $H=H^{1}_{0}(\Omega,\RR)$ 
by standard finite elements.
For an elliptic bilinear form $a(\cdot,\cdot)$ and given $f\in H^{-1}$
consider the energy $J(v)=\frac{1}{2}a(v,v) - (f,v)$, the variational equalities
\begin{gather*}
\begin{aligned}
u&\in H: &\quad a(u,v)&=(f,v) &\quad \forall v&\in H,\\
u_{h}&\in S_{h}^{m}: &\quad a(u_{h},v_{h})&=(f,v_{h}) &\quad \forall v_{h}&\in S_{h}^{m},\\
\end{aligned}
\end{gather*}
and the adjoint problem
\begin{align*}
w\in H:\qquad a(v,w)&=(g,v)\qquad \forall v\in H ,
\end{align*}
where $g\colonequals u-u_{h}$.
We assume $H^2$-regularity of the adjoint problem, i.e., $|w|_{H^2}\leq C\|g\|_{L^2}$.
Using Galerkin orthogonality and the $H^1$-ellipticity of $a(\cdot,\cdot)$, we can then estimate
\begin{align*}
\|u-u_{h}\|^{2}_{L^2}&=(g,u-u_{h})=a(u-u_{h},w)= a(u-u_{h},w-w_{I})\\
&\leq \Lambda\|u-u_{h}\|_{H^{1}}\|w-w_{I}\|_{H^{1}}\\
&\leq C h^{m} |u|_{H^{k}}\; h\;|w|_{H^{2}}\\
&\leq C h^{m+1} |u|_{H^2} \|u-u_{h}\|_{L^2}.
\end{align*}
We want to generalize the Aubin--Nitsche lemma to the $L^2$-error for energies for functions with manifold targets. However, the proof given above only works for quadratic energies with linear target space, i.e., for energies with linear Euler--Lagrange equations.
Euclidean techniques to obtain error estimates for nonlinear energies as found in \cite{Rannacher} rely on the deformation to a linear problem and weighted norms.
A generalization of these is desirable, in particular since they also provide $L^{\infty}$-error estimates.
We, however, follow the general approach to geometrically generalize the concept of linearity rather than use a linearization.
We restrict our analysis to energies that are ``predominantly quadratic", by which we mean the following bound on the third variation of the energy.
% %
\begin{definition}\label{def:almostLinear}
	Let $q>\max\{d,2\}$ and $\Energy:H \to \RR$ be an energy functional.
	We say that $\Energy$ is predominantly quadratic if $\Energy$ is $C^{3}$ along geodesic homotopies, 
	and for any $v\in H \cap W^{1,q}_{K}$, and vector fields $U,V$ along $v$ 
	\begin{align}\label{eq:almostLinear}
	|\delta^{3}\Energy(v)(U,V,V)|\leq C(K,M)\|U\|_{W^{2,2}(\Omega,v^{-1}TM)}\|V\|_{W^{1,2}(\Omega,v^{-1}TM)}
	\|V\|_{W^{o,r}(\Omega,v^{-1}TM)},
	\end{align}
	with either $(o,r)=(1,2)$, or $o=0$ and $r\leq d$.
\end{definition}
Note that this is still a restriction on the energy, not on the manifold $M$.
A combination with a deformation argument in order to obtain error estimates for more general energies is conceivable, but beyond the scope of this work.
\begin{example}
	In the Euclidean case $M=\RR^{n}$, quadratic energies are predominantly quadratic, as the third variation vanishes.
	As long as the coefficient functions of a semilinear PDE coming from a minimization problem are smooth enough and bounded, 
	the third variation of the energy will have a bound of the form
	\begin{align*}
	|\delta^{3}\Energy(v)(U,V,V)|\leq C\;\int_{\Omega}(|U|(|\nabla V|+|V|)^{2}+|\nabla U||\nabla V||V|)\;dx.
	\end{align*}
	Such an energy is also predominantly quadratic.
	
	The leading term of the third variation of the energy for a typical quasi-linear equation, 
	e.g., the minimal surface energy for graphs $\Energy(u)=\int_{\Omega}\sqrt{1+|Du|^{2}}\;dx$, has the form
	\begin{align*}
	|\delta^{3}\Energy(v)(U,V,V)|\leq C\;\int_{\Omega}|\nabla U|\;|\nabla V|^{2}\;dx.
	\end{align*}
	For $d=1$ such an energy is predominantly quadratic, but in general not for higher dimensions.
\end{example}
For a Riemannian manifold $M$, the harmonic energy is predominantly quadratic (see Section~\ref{sec:ch4}).

\subsection{Galerkin Orthogonality}
We consider the variational formulation of the problems \eqref{eq:Pcont} and \eqref{eq:Pdisc}
\begin{align}\label{eq:PcontVar}
u\in H:\qquad \frac{d}{dt}\bigg|_{t=0}\Energy(\exp_{u}(tV))=0\qquad \forall V\in W_{0}^{1,2}(\Omega,u^{-1}TM),
\end{align}
and
\begin{align}\label{eq:PdiscVar}
u_{h}\in S_{h}^{m}:\qquad \frac{d}{dt}\bigg|_{t=0}\Energy(\exp_{u_{h}}(tV_{h}))=0\qquad \forall V_{h}\in IV_{0}(\Omega,u_{h}^{-1}TM).
\end{align}
In the context of quadratic energies in the Euclidean setting, Galerkin orthogonality states that the error between the continuous and the discrete solution is perpendicular to the discrete space with respect to the quadratic form.

As we are in a nonlinear setting, the natural bilinear form is the second variation of the energy. The discrete test vector fields are a priori only defined at the discrete solution which suggests a orthogonality principle of the form
\begin{align}\label{eq:suggestOrtho}
\delta^{2}\Energy(u_{h})(\log_{u_{h}}u,V_{h})=0\qquad \forall V_{h}\in  IV_{0}(\Omega,u_{h}^{-1}TM).
\end{align}
However, we cannot use~\eqref{eq:PcontVar} to show this kind of orthogonality without transporting the discrete test vector to $u$.
The following proposition shows the impact the choice of transport has on the notion of Galerkin orthogonality one can obtain from~\eqref{eq:PcontVar} and~\eqref{eq:PdiscVar}.
\begin{proposition}\label{prop:galerkin}
	Let $u$ and $u_{h}$ be solutions to \eqref{eq:PcontVar} and \eqref{eq:PdiscVar}, respectively, 
	and let $\Gamma$ be the geodesic homotopy joining $u$ and $u_{h}$.
	Then, for any transport $V_{h}:I\to W_{0}^{1,2}(\Omega,\Gamma(t)^{-1}TM)$ along $\Gamma$ of any discrete vector field $V_{h}(1)=V_{h,1}\in IV(\Omega,u_{h}^{-1}TM)$
	holds
	\begin{align}\label{eq:galerkin}
	\int_{0}^{1}\delta^{2} \Energy(\Gamma(t))(V_{h}(t),\dot{\Gamma}(t)) + \delta \Energy(\Gamma(t))(\nabla_{t}V_{h}(t))\;dt=0.
	\end{align}
\end{proposition}
\begin{proof}
	Indeed, as $u$ and $u_{h}$ fulfill \eqref{eq:PcontVar} and \eqref{eq:PdiscVar}, 
	respectively, we have 
	\begin{align*}
	\int_{0}^{1}\delta^{2} \Energy(\Gamma(t))(V_{h}(t),\dot{\Gamma}(t))+\delta \Energy(\Gamma(t))(\nabla_{t}V_{h}(t))\;dt
	&= \int_{0}^{1} \frac{d}{dt} \delta \Energy(\Gamma(t)(V_{h}(t))\;dt\\
	&= \delta \Energy(u_{h})(V_{h}(1))-\delta \Energy(u)(V_{h}(0))\\
	&=0.\qedhere
	\end{align*}
\end{proof}
To obtain~\eqref{eq:suggestOrtho} from this, one can define a transport with the initial values $V_{h}(1)=V_{h,1}$ and $\nabla_{t}V_{h}(1)=0$ such that the integrand in~\eqref{eq:galerkin} is constant in $t$, i,e., $V_{h}$ fulfills the second order equation
\begin{align*}
\delta\Energy(\Gamma(t))(\nabla^{2}_{t}V_{h}(t))
+ 2\delta^{2}\Energy(\Gamma(t))(\dot \Gamma(t),\nabla_{t}V_{h}(t))
+ \delta^{3}\Energy(\Gamma(t))(\dot \Gamma(t), \dot \Gamma(t),V_{h}(t))\equiv 0.
\end{align*}
However, as this system is rather complicated, we will in the following use~\eqref{eq:galerkin} with parallel transport as our notion of Galerkin orthogonality instead of~\eqref{eq:suggestOrtho}.
\subsection{The Deformation Problem}
We now define a nonlinear generalization of the adjoint problem from the Aubin--Nitsche-Trick.
For nonlinear energies the adjoint problem is essentially a linearization of problem~\eqref{eq:PcontVar}
with a right hand side that is given by the difference of the solutions $u$ and $u_{h}$ to~\eqref{eq:PcontVar} and~\eqref{eq:PdiscVar}, respectively.

In the context of Riemannian manifolds, this linearized problem will not act on functions but on vector fields:

Find $(w,W)\in W^{1,2}(\Omega,TM)$ such that
\begin{align}\label{eq:deformedCont}
\delta^{2}\Energy(w)(W,V_{1}) + \delta \Energy(w)(V_{2})=-(V_{1},\log_{w}u_{h} - \log_{w}u)_{L^{2}(\Omega,w^{-1}TM)}
\end{align}
holds for all tangent vectors $(V_{1},V_{2})\in T_{(w,W)}TH=(W^{1,2}_{0}(\Omega,w^{-1}TM))^{2}$.

\begin{remark}
	One can easily check by inserting test vector fields of the form $(V_{1},0)$ that the solution $(w,W)$ 
	of \eqref{eq:deformedCont} projects over the solution $u$ of \eqref{eq:Pcont}, i.e., $w=u$.
	Thus, \eqref{eq:deformedCont} is equivalent 
	to the system consisting of \eqref{eq:PcontVar} and finding $W\in W^{1,2}(\Omega,u^{-1}TM)$ such that
	\begin{align}\label{eq:deformedContVec}
	\delta^{2}\Energy(u)(W,V) =-(V,\log_{u}u_{h})_{L^{2}(\Omega,u^{-1}TM)}\qquad \forall V\in W^{1,2}(\Omega,u^{-1}TM).
	\end{align}
	As the bilinear form acts on deformations of $u$, we call \eqref{eq:deformedContVec} the deformation problem.
\end{remark}
\subsection{Parallel Transport of Vector Fields}
One major difference to the Euclidean setting is that the base points of vector fields do not matter there.
In order to use Galerkin orthogonality in the manifold setting, we however need a discrete vector field along $u_{h}$ as a valid test vector field for~\eqref{eq:PdiscVar}.
The discretization of~\eqref{eq:deformedCont} projects over $u_{h}$, but discretization error bounds are not trivial, as a best approximation of the solution $W\in W^{1,2}(\Omega,u^{-1}TM)$ of~\eqref{eq:deformedContVec} in $IV(\Omega,u_{I}^{-1}TM)$ is a discrete vector field not along $u_{h}$ but along $u_{I}$.
Thus, we instead first parallel transport the $W$ to $u_{h}$ along a geodesic homotopy and then approximate along $u_{h}$.

In order to preserve bounds through this transport, we will need the following technical estimate which generalizes the Uniformity Lemma~\ref{L:uniformity}.
\begin{proposition}\label{L:thetaPara}
	Let $q>\max\{d,4\}$, $a,b$ as in Definition~\ref{def:nonlinSmoothVfield} with $p=2$ and $k=2$, and $\Gamma$ be a geodesic homotopy such that $\Gamma(s)\in H$ with $\Gamma(s)_{|T_{h}} \in W^{2,b}(T_{h},M)$ for all $s\in [0,1]$ and elements $T_{h}\in G$.
	We set $K_{1}=\max_{s}\dot{\theta}_{1,q,\Omega}(\Gamma(s))$, and $K_{2}=\max_{s}\dot{\theta}_{2,b,G}(\Gamma(s))$.
	\\
	Then there exists a constant $\Cl{c:thetaPara}$ depending on $K_{1}$ and $M$, but independent of $K_{2}$ such that
	\begin{align}
	\Theta_{2,2,G}(W(1))
	\leq \Cr{c:thetaPara}\left(\|W(0)\|_{W^{2,2}(\Gamma(0)^{-1}TM)}
	+ K_{2}\|W(0)\|_{L^{a}(\Omega,\Gamma(0)^{-1}TM)}\right)
	\end{align}
	holds for the parallel transport $W(s)$ of any vector field $W(0)\in W^{2,2}(\Omega,\Gamma(0)^{-1}TM)$.
\end{proposition}

\begin{proof}
	We write
	\begin{align*}
	\Theta_{2,2,G}^{2}(W(1))
	&=\sum_{T_{h}\in G}	\Theta_{2,2,T_{h}}^{2}(W(1))\\
	&\!\begin{multlined}[t][\displaywidth]
	=\sum_{T_{h}\in G}\bigg(\dot\Theta_{0,2,T_{h}}^{2}(W(1)) +\dot\Theta_{1,2,T_{h}}^{2}(W(1))
	+\|W(1)\|_{L^{a}(T_{h},\Gamma(1)^{-1}TM)}^{2}\dot\theta_{2,b,T_{h}}^{2}(\Gamma(1))\\
	+\int_{T_{h}}|\nabla W(1)|^{2}|d\Gamma(1)|^{2}\;dx
	+\|\nabla^{2} W(1)\|_{L^{2}(T_{h},\Gamma(1)^{-1}TM)}^{2}
	\bigg).
	\phantom{\Theta_{2,2,G}^{2}(W(1))}
	\end{multlined}
	\end{align*}
	By H\"older's inequality we have
	\begin{align*}
	\Theta_{2,2,G}(W(1))&\leq
	\dot\Theta_{0,2,\Omega}(W(1)) +\dot\Theta_{1,2,\Omega}(W(1))
	+\|W(1)\|_{L^{a}(\Omega,\Gamma(1)^{-1}TM)}\dot\theta_{2,b,G}(\Gamma(1))\\
	&\quad+\|\nabla W(1)\|_{L^{\frac{2q}{q-2}}(\Omega,\Gamma(1)^{-1}TM)}\dot\theta_{1,q,\Omega}(\Gamma(1))
	+\|\nabla^{2} W(1)\|_{L^{2}(G,\Gamma(1)^{-1}TM)}.
	\end{align*}
	As $\frac{2q}{q-2}<q$ and $b_{1}\leq q$, we can use Lemma~\ref{L:uniformity} to estimate
	\begin{align*}
	\Theta_{2,2,G}(W(1))&\leq
	\|W(0)\|_{L^{2}(\Omega,\Gamma(0)^{-1}TM)}
	+(K_{1}+\Cr{c:uniformity}+1)\|W(0)\|_{W^{1,2}(\Omega,\Gamma(0)^{-1}TM)}\\
	&\quad +K_{2}\|W(0)\|_{L^{a}(\Omega,\Gamma(0)^{-1}TM)}+K_{1}(\Cr{c:uniformity}+1)\|W(0)\|_{W^{1,\frac{2q}{q-2}}(\Omega,\Gamma(0)^{-1}TM)}\\
	&\quad
	+\|\nabla^{2} W(1)\|_{L^{2}(G,\Gamma(1)^{-1}TM)}.
	\end{align*}
	Note that the choice of $q$ and $a$ allows for the following Sobolev--type estimate
	\begin{align*}
	\|V\|_{L^{a}} + \|V\|_{W^{1,\frac{2q}{q-2}}}&\leq C\; \|V\|_{W^{2,2}}.
	\end{align*}
	Thus, all that is left to estimate are the second order derivatives of $W$.
	As in the proof of Lemma~\ref{L:uniformity}, we differentiate and use H\"older's inequality to obtain
	\begin{align*}
	\frac{d}{dt}& \|\nabla^{2} W(t)\|_{L^{2}(G,\Gamma(t)^{-1}TM)}\\
	&=\|\nabla^{2} W(t)\|_{L^{2}(G,\Gamma(t)^{-1}TM)}^{-1}
	\sum_{T_{h}\in G}\int_{T_{h}}\langle\nabla_{t}\nabla_{x}^{2}W(t),\nabla_{x}^{2}W(t)\rangle_{g}\;dx\\
	&\!\begin{multlined}[t][\displaywidth]\leq C(M)\|\nabla^{2} W(t)\|_{L^{2}(G,\Gamma(t)^{-1}TM)}^{-1}\\
	\sum_{T_{h}\in G}\int_{T_{h}}|\nabla^{2}W|\left(
	|\dot\Gamma|\;|d\Gamma|\;|\nabla W| + |\nabla\dot\Gamma|\;|d\Gamma|\;|W| + |\dot\Gamma|\;|\nabla^{2} \Gamma|\;|W|
	\right)\;dx
	\phantom{\frac{d}{dt}}
	\end{multlined}\\
	&\leq \!\begin{multlined}[t][\displaywidth]
	C(M)
	\bigg(
	\dot\theta_{1,q,\Omega}(\Gamma(t))\|\dot\Gamma(t)\|_{L^{\infty}(\Omega,\Gamma(t)^{-1}TM)}
	\|\nabla W(t)\|_{L^{\frac{2q}{q-2}}(\Omega,\Gamma(t)^{-1}TM)}\\
	\begin{aligned}
	&+\dot\theta_{1,q,\Omega}(\Gamma(t))\|\nabla\dot\Gamma(t)\|_{L^{q}(\Omega,\Gamma(t)^{-1}TM)}
	\|W(t)\|_{L^{\frac{2q}{q-4}}(\Omega,\Gamma(t)^{-1}TM)}\phantom{\frac{d}{dt}}\\
	&+\dot\theta_{2,b,G}(\Gamma(t))\|\dot\Gamma(t)\|_{L^{\infty}(\Omega,\Gamma(t)^{-1}TM)}
	\|W(t)\|_{L^{a}(\Omega,\Gamma(t)^{-1}TM)}
	\bigg).
	\end{aligned}
	\phantom{\frac{d}{dt}}
	\end{multlined}
	\end{align*}
	Using Lemma~\ref{L:uniformity}, we obtain
	\begin{multline*}
	\frac{d}{dt} \|\nabla^{2} W(t)\|_{L^{2}(G,\Gamma(t)^{-1}TM)}\\
	\leq C(M) K_{1} \left(
	K_{1}(\Cr{c:uniformity}+1)
	\|W(0)\|_{W^{2,2}(\Omega,\Gamma(0)^{-1}TM)}
	+K_{2}
	\|W(0)\|_{L^{a}(\Omega,\Gamma(0)^{-1}TM)}
	\right),
	\end{multline*}
	and thus
	\begin{align*}
	\|\nabla^{2}W(1)\|_{L^{2}(\Omega,\Gamma(t)^{-1}TM)}
	&= \|\nabla^{2} W(0)\|_{L^{2}(\Omega,\Gamma(0)^{-1}TM)}+\int_{0}^{1}\frac{d}{dt} \|\nabla^{2}W(t)\|_{L^{2}(\Omega,\Gamma(t)^{-1}TM)}\;dt\\
	&\leq C\;\left(\|W(0)\|_{W^{2,2}(\Omega,\Gamma(0)^{-1}TM)}+K_{2} \|W(0)\|_{L^{a}}\right),
	\end{align*}
	which yields the assertion.
\end{proof}
\subsection{$L^{2}$-Error Estimate}
We can now prove the $L^{2}$-error estimate, which is the main result of this chapter.
The statement of the theorem contains a discrete $H^{2}$-regularity assumption that will be discussed in Section~\ref{sec:discH2reg}.
\begin{theorem}\label{T:L2err}
	Let the assumptions of Theorem~\ref{T:H1err} and Lemma~\ref{L:chooseq} be fulfilled with $q>\max\{d,4\}$ chosen as in Lemma~\ref{L:chooseq} for a predominantly quadratic energy $\Energy$, and a discrete set $S_{h}^{m}$ fulfilling Condition~\ref{cond:2b}. 
	\\
	Let $u_{h}$ be a minimizer of $\Energy$ in $S_{h}^{m}\cap H_{K,L}$
	under the boundary and homotopy conditions.
	We assume that
	\begin{align} \label{eq:2bBound}
	\theta_{2,b,G}(u_{h})\leq K_{2}
	\end{align}
	for a constant $K_{2}$.
	Finally, suppose that the deformation problem~\eqref{eq:deformedContVec} is $H^{2}$-regular, i.e., that the solution $W$ fulfills
	\begin{align}\label{eq:H2reg}
	\|W\|_{W^{2,2}(\Omega,u^{-1}TM)}\leq C\;\|\log_{u}u_{h}\|_{L^{2}(\Omega,u^{-1}TM)}.
	\end{align}
	Then there exists a constant $\Cl{c:L2const}$, such that
	\begin{align*}
	d_{L^{2}}(u,u_{h})\leq \Cr{c:L2const}\;h^{m+1} \left(\theta_{m+1,2,\Omega}(u)+\theta^{2}_{m+1,2,\Omega}(u)\right).
	\end{align*}
\end{theorem}
\begin{proof}
	First note that it is indeed possible to choose $4<q<r$ with $r$ as in the proof of Lemma~\ref{L:chooseq},
	as either
	\begin{enumerate}
		\item $2m<d<2(m+1)$, $m\geq2$, and $r=\frac{2d}{d-2m}>2m\geq 4$,
		\item or $m=1$ and $d=3$, and $	r=6>4$,
		\item or  $r$ can be arbitrarily large.
	\end{enumerate}
	We insert $V\colonequals \log_{u}u_{h}$ into \eqref{eq:deformedContVec}, and obtain
	\begin{align*}
	d_{L^{2}}^{2}(u,u_{h})&= -\delta^{2}\Energy(u)(W,\log_{u}u_{h}),
	\end{align*}
	where $W\in W^{2,2}(\Omega,u^{-1}TM)$ is the solution of \eqref{eq:deformedContVec}.
	
	Let $\Gamma$ denote the geodesic homotopy joining $u$ and $u_{h}$, and $W(t)$ the parallel transport of $W$ along $\Gamma$.
	Let $W(1)_{I}$ be the approximation of $W(1)$ along $u_{h}$ in the sense of Condition~\ref{cond:2b}, and let $W_{I}(t)$ denote its parallel transport along $\Gamma$.
	
	As Lemma~\ref{L:chooseq} ensures that $u_{h}$ is a local minimizer in $S_{h}^{m}$, generalized Galerkin orthogonality (Proposition~\ref{prop:galerkin}) holds, so that
	\begin{align}
	d^{2}_{L^{2}}(u,u_{h})&=  - \delta^{2}\Energy(u)(W,\log_{u}u_{h}) + \int_{0}^{1}\delta^{2} \Energy(\Gamma(t))\left(W_{I}(t),\dot{\Gamma}(t)\right)\;dt \nonumber \\
	&= \int_{0}^{1} \int_{0}^{t} \frac{d}{ds} \delta^{2} \Energy(\Gamma(s))\left(\frac{s}{t} W_{I}(s) + \left(1-\frac{s}{t}\right)W(s),\dot{\Gamma}(s)\right)\;ds\;dt \nonumber \\
	&= \int_{0}^{1} \int_{0}^{t}  \delta^{3} \Energy(\Gamma(s))\left(\frac{s}{t} W_{I}(s)+\left(1-\frac{s}{t}\right)W(s),\dot{\Gamma}(s),\dot{\Gamma}(s)\right)\;ds\;dt \nonumber \\
	&\qquad  + \int_{0}^{1} \int_{0}^{t} \frac{1}{t} \delta^{2} \Energy(\Gamma(s))(W_{I}(s)- W(s),\dot{\Gamma}(s))\;ds\;dt. \label{eq:dummieGalerkin}
	\end{align}
	We can estimate the second integral in \eqref{eq:dummieGalerkin} using the ellipticity assumption \eqref{eq:ellipticabove}
	\begin{multline*}
	\int_{0}^{1} \int_{0}^{t} \frac{1}{t} \delta^{2} \Energy(\Gamma(s))(W_{I}(s)- W(s),\dot{\Gamma}(s))\;ds\;dt\\
	\leq \Lambda \int_{0}^{1} \int_{0}^{t} \frac{1}{t} \|W_{I}(s)- W(s)\|_{W^{1,2}(\Omega,\Gamma(s)^{-1}TM)} \|\dot{\Gamma}(s)\|_{W^{1,2}(\Omega,\Gamma(s)^{-1}TM)}\;ds\;dt.
	\end{multline*}
	As the vector fields $W_{I}$, $W$, and $\dot{\Gamma}$ are parallel along $\Gamma$, we can further estimate using the Uniformity Lemma~\ref{L:uniformity}
	\begin{multline*}
	\int_{0}^{1} \int_{0}^{t} \frac{1}{t} \delta^{2} \Energy(\Gamma(s))(W_{h}(s)- W(s),\dot{\Gamma}(s))\;ds\;dt\\
	\leq C\;\|W(1)_{I}- W(1)\|_{W^{1,2}(\Omega,u_{h}^{-1}TM)} \|\log_{u}u_{h}\|_{W^{1,2}(\Omega,u^{-1}TM)}.
	\end{multline*}
	Condition~\ref{cond:2b}, Proposition~\ref{L:thetaPara}, and Assumption \eqref{eq:2bBound} imply that
	\begin{align*}
	\|W(1)_{I}- W(1)\|_{W^{1,2}(\Omega,u_{h}^{-1}TM)}
	&=\left(\sum_{T_{h}\in G}\|W(1)_{I}- W(1)\|^{2}_{W^{1,2}(T_{h},u_{h}^{-1}TM)}\right)^{\frac{1}{2}}\\
	&\leq C\;h\;\Theta_{2,2,G}(W(1))\\
	&\leq C\;h\; \|W\|_{W^{2,2}( \Omega,u^{-1}TM)}.
	\end{align*}
	Combining this with the $H^{2}$-regularity and Theorem~\ref{T:H1err}, we obtain
	\begin{align*}
	\int_{0}^{1} \int_{0}^{t} \frac{1}{t} \delta^{2} \Energy(\Gamma(s))(W_{I}(s)- W(s),\dot{\Gamma}(s))\;ds\;dt
	&\leq C\;h^{m+1}\;d_{L^{2}}(u,u_{h})\;\theta_{m+1,2,\Omega}(u).
	\end{align*}
	In order to estimate the first integral term in \eqref{eq:dummieGalerkin} we use that $\Energy$ is predominantly quadratic
	\begin{multline*}
	|\delta^{3} \Energy(\Gamma(s))(s W_{I}(s)+(t-s)W(s),\dot{\Gamma}(s),\dot{\Gamma}(s))|\\*
	\leq C(K,M) \|\dot{\Gamma}(s)\|_{W^{1,2}}\|\dot{\Gamma}(s)\|_{W^{o,r}} \left(s \|W_{I}(s)\|_{W^{2,2}} + (t-s)\|W(s)\|_{W^{2,2}}\right).
	\end{multline*}
	The Uniformity Lemma~\ref{L:uniformity} and Theorem~\ref{T:H1err} imply further that for all $s\in[0,1]$
	\begin{align*}
	\|\dot{\Gamma}(s)\|_{W^{1,2}}\|\dot{\Gamma}(s)\|_{W^{o,r}}\leq C\; \|\log_{u}u_{h}\|_{W^{1,2}}\|\log_{u}u_{h}\|_{W^{o,r}}\leq C\; h^{m}\theta_{m+1,2,\Omega}(u)\;\|\log_{u}u_{h}\|_{W^{o,r}}.
	\end{align*}
	Using first Proposition~\ref{L:thetaPara}, then Condition~\ref{cond:2b} and then again Proposition~\ref{L:thetaPara} we obtain for all $s\in [0,t]$ 
	\begin{align*}
	s \|W_{I}(s)\|_{W^{2,2}} + (t-s)\|W(s)\|_{W^{2,2}}
	& \leq C\; s \|W_{I}(1)\|_{W^{2,2}} + (t-s)\|W(s)\|_{W^{2,2}}\\
	& \leq C\;s  \Theta_{2,2,G}(W(1)) + (t-s)\|W(s)\|_{W^{2,2}}\\
	&\leq C\;t\|W(0)\|_{W^{2,2}}.
	\end{align*}
	Thus we obtain, using the $H^{2}$-regularity of the deformation problem \eqref{eq:deformedContVec},
	\begin{multline*}
	\int_{0}^{1} \int_{0}^{t}  \delta^{3} \Energy(\Gamma(s))\left(\frac{s}{t} W_{I}(s)+\left(1-\frac{s}{t}\right)W(s),\dot{\Gamma}(s),\dot{\Gamma}(s)\right)\;ds\;dt\\*
	\leq C\;h^{m}\theta_{m+1,2,\Omega}(u)\;\|\log_{u}u_{h}\|_{W^{o,r}}\; d_{L^{2}}(u,u_{h}).
	\end{multline*} 
	If $o=1$ and $r=2$,
	then
	\begin{multline*}
	\int_{0}^{1} \int_{0}^{t}  \delta^{3} \Energy(\Gamma(s))\left(\frac{s}{t} W_{I}(s)+\left(1-\frac{s}{t}\right)W(s),\dot{\Gamma}(s),\dot{\Gamma}(s)\right)\;ds\;dt\\*
	\leq C\;h^{2m}\theta^{2}_{m+1,2,\Omega}(u)\; d_{L^{2}}(u,u_{h}),
	\end{multline*}
	with $2m\geq m+1$.
	
	If instead $o=0$ and $r\leq d$, then either we are in the same situation as before, or $d\geq4$ and $\frac{2d}{d-2}\leq r\leq d$. In that case $L^{p}$-interpolation with $\epsilon=h$ yields
	\begin{align*}
	\|\log_{u}u_{h}\|_{L^{r}}&\leq h\; \|\log_{u}u_{h}\|_{L^{\infty}} + h^{1-\frac{r(d-2)}{2d}}\|\log_{u}u_{h}\|_{L^{\frac{(d-2)}{2d}}}\\
	&\leq C\; h + C\; h^{m+1-\frac{r(d-2)}{2d}}\theta_{m+1,2,\Omega}(u).
	\end{align*}
	As $(m+1)\geq \frac{d}{2}$ and $d\geq r$, we have $m+1-\frac{r(d-2)}{2d} \geq 1$.
	Thus, we obtain also for this case
	\begin{align*}
	\int_{0}^{1} \int_{0}^{t}  \delta^{3} \Energy(\Gamma(s))\left(\frac{s}{t} W_{I}(s)+\left(1-\frac{s}{t}\right)W(s),\dot{\Gamma}(s),\dot{\Gamma}(s)\right)\;ds\;dt
	\leq C\;h^{m+1}\; d_{L^{2}}(u,u_{h}).
	\end{align*}
	This yields the assertion.
\end{proof}
\subsection{Discrete $H^{2}$-Regularity}\label{sec:discH2reg}
Theorem~\ref{T:L2err} assumes \eqref{eq:2bBound}, i.e., a bound on second derivatives of the discrete solution $u_{h}$.
In Lemma~\ref{L:chooseq}, Condition~\ref{cond:3} and convergence in $D_{1,2}$ is used to show that the discrete solution stays away from the additional bounds $K$ and $L$ in $H^{1,2,q}_{K,L}$. Now, we will show that under Condition~\ref{cond:4} also Assumption~\eqref{eq:2bBound} is implied.

In contrast to Assumption~\eqref{eq:2bBound}, Conditions~\ref{cond:3} and~\ref{cond:4} are general assumptions on inverse estimates on the discrete functions. We will introduce a class of finite elements in Section~\ref{sec:ch2} that fulfill these conditions.
\begin{proposition}\label{prop:discreteH2Reg}
	Let $m\geq 1$ and $2(m+1)>d$.
	Define $a,b$ as in Definition~\ref{def:nonlinSmoothVfield} with $p=k=2$,
	and let $q>\max\{4,d\}$ fulfill
	\begin{align*}
	2b\leq q< \left\{\begin{array}{ll}
	\infty &\ \textrm{for}\ d<2m\\
	\epsilon^{-1} &\ \textrm{for}\ 2=dm\\
	\frac{2d}{d-2m} &\ \textrm{for}\ 2m<d<2(m+1).
	\end{array}\right.
	\end{align*}
	Suppose that the grid $G$ on $\Omega$ is of width $h$ and order $m$, that $F^{-1}_{h}:T\to T_{h}$ scales with order $2$ for all elements $T_{h}$, and that $S_{h}^{m}$ fulfills Conditions~\ref{cond:1},~\ref{cond:3}, and~\ref{cond:4}.\\
	For $v\in W^{m+1,2}(\Omega,M)$, and $v_{h}\in S^{m}_{h}\cap  H^{1,2,q}_{K,L}$ with $v_{I}|_{\partial\Omega}=v_{h}|_{\partial\Omega}$, we assume the relation
	\begin{align}\label{eq:H1errorAssumed}
	D_{1,2}(v,v_{h})\leq C\;h^{m}\theta_{m+1,2,\Omega}(v).
	\end{align}
	Then there exists a constant $K_{2}$ depending on $v$ and $K$ but independent of $h$ such that
	\begin{align}\label{eq:discRegd>1}
	\theta_{2,b,G}(v_{h})\leq K_{2}
	\end{align}
	if $h$ is small enough.
\end{proposition}

\begin{proof}
	First note that it is possible to choose $q$ as proposed. Indeed if $2m<d<2(m+1)$, we distinguish   three cases.
	\begin{enumerate}
		\item For $d<3$ and $d=4$, there is no $m\geq 1$ with $2m<d<2(m+1)$.
		\item For $d=3$, the condition $2m<d<2(m+1)$ holds only for $m=1$. In that case $2b=4<6=\frac{2d}{d-2m}$.
		\item For $d> 4$, the condition $d<2(m+1)$ implies $2b=d<\frac{2d}{d-2m}$.
	\end{enumerate}
	As $b\leq q$ it is enough to estimate the homogeneous part $\dot{\theta}_{2,b,G}(v_{h})$.
	By the choice of $q$, Lemma~\ref{L:chooseq} shows that for $h$ small enough
	\begin{align*}
	d_{L^{\infty}}(v,v_{h})\leq C\;h^{\delta}\leq \rho
	\end{align*}
	even for small $\rho$. By the triangle inequality this extends to the approximation $v_{I}$ of $v$ defined by Condition~\ref{cond:1}. 
	Let $\Gamma$ denote the geodesic homotopy connecting $v_{I}$ to $v_{h}$. Then
	\begin{align*}
	\dot\theta_{2,b}(v_{h})
	&\leq C\;\theta_{1,2b}^{2}(v_{h}) + C\;\dot\theta_{2,b}(v_{I})+ C\int_{0}^{1}\frac{d}{dt}\|\nabla^{2} \Gamma\|_{L^{b}}\;dt\\
	&
	\!\begin{multlined}[t][\displaywidth]
	= C\;\theta_{1,2b}^{2}(v_{h}) + C\;\dot\theta_{2,b}(v_{I})\\
	+ C\int_{0}^{1}\|\nabla^{2} \Gamma\|_{L^{b}}^{1-b}\int_{\Omega}|\nabla^{2}\Gamma|_{g}^{b-2}\left(\langle \nabla^{2}\dot\Gamma,\nabla^{2}\Gamma\rangle_{g}+\Rm(\dot \Gamma,d\Gamma,d\Gamma,\nabla^{2}\Gamma)\right)\;dx\;dt
	\phantom{\dot\theta_{2,b}(v_{h})}
	\end{multlined}
	\\
	&\leq  C\;\theta_{1,2b}^{2}(v_{h}) + C\;\dot\theta_{2,b}(v_{I})+ 
	C\int_{0}^{1}\|\nabla^{2}\dot\Gamma\|_{L^{b}} +\rho \theta_{1,2b}^{2}(\Gamma(t))\;dt.
	\end{align*}
	Now $\theta_{1,2b}^{2}(\Gamma(t))\leq C\max\{\theta_{1,2b}^{2}(v_{h}),\theta_{1,2b}^{2}(v_{I})\}$.
	We use Proposition~\ref{L:thetaPara}
	to estimate
	\begin{align*}
	\int_{0}^{1}\|\nabla^{2}\dot\Gamma(t)\|_{L^{b}}\;dt&\leq C\;\|\nabla^{2}\dot\Gamma(0)\|_{L^{b}} + C\;\left(\dot\theta_{2,b}(v_{h})+\dot\theta_{2,b}(v_{I})\right)\|\dot\Gamma(0)\|_{L^{b}}\\
	&\leq C\;\|\nabla^{2}\dot\Gamma(0)\|_{L^{b}} +  C\;\left(\dot\theta_{2,b}(v_{h})+\dot\theta_{2,b}(v_{I})\right)\rho.
	\end{align*}
	Thus, if $\rho$ is small enough, we have
	\begin{align*}
	\dot\theta_{2,b}(v_{h})\leq C\;\theta_{1,2b}^{2}(v_{h}) + C\;\dot\theta_{2,b}(v_{I}) + C\;\|\nabla^{2}\log_{v_{I}}v_{h}\|_{L^{b}}.
	\end{align*}
	Condition~\ref{cond:4} implies
	\begin{align*}
	\dot\theta_{2,b}(v_{h})
	&\leq 
	C\;\theta_{1,2b}^{2}(v_{h}) + C\;\dot\theta_{2,b}(v_{I})+ 
	C h^{-1-d\left(\frac{1}{b}-\frac{1}{2}\right)}D_{1,2}(v_{h},v_{I}) + C h^{-2} d_{L^{b}}(v_{h},v_{I})\\
	&\leq 
	C\;K^{2} + C\;\dot\theta_{2,b}(v)+ 
	C h^{m-1-d\left(\frac{1}{b}-\frac{1}{2}\right)}\theta_{m+1,2}(v) + C h^{-2} d_{L^{b}}(v_{h},v_{I}).
	\end{align*}
	Note that for $m=1$, we have
	\begin{align*}
	m-1-d\left(\frac{1}{b}-\frac{1}{2}\right)=0,
	\end{align*}
	and the $O(h^{-2})$-term vanishes, which yields the assertion.
	Otherwise, $m>\max\{1,\frac{d-2}{2}\}$ implies
	\begin{align*}
	m-1-d\left(\frac{1}{b}-\frac{1}{2}\right) \geq \left\{
	\begin{array}{ll}
	1 &\quad\textrm{for}\ d<4\\
	\frac{1}{3}  &\quad\textrm{for}\ d=4\\
	\frac{2d-8}{2} &\quad\textrm{for}\ d>4,
	\end{array}
	\right\}>0.
	\end{align*}
	For $d\leq 6$, using the Poincar\'e inequality, the approximate triangle inequality for $D_{1,2}$, and Condition~\ref{cond:1}, we have
	\begin{align*}
	h^{-2}d_{L^{b}}(v_{h},v_{I})\leq C\;h^{-2}\;D_{1,2}(v_{h},v_{I})\leq C\;h^{m-2}\theta_{m+1,2}(v).
	\end{align*}
	For $d>6$, we obtain using Condition~\ref{cond:4} and the Poincar\'e inequality
	\begin{align*}
	h^{-2}d_{L^{b}}(v_{h},v_{I})&\leq C h^{-5+\frac{d}{2}}d_{L^{\frac{2d}{d-2}}}(v_{h},v_{I})\leq C\;h^{-5+\frac{d}{2}}\;D_{1,2}(v_{h},v_{I})\leq C\;h^{m-5+\frac{d}{2}}\theta_{m+1,2}(v)\\
	&\leq C\;h^{d-6}\theta_{m+1,2}(v).
	\end{align*}
	This proves the bound on $\theta_{2,b,G}(u_{h})$.
\end{proof}
% % % % % % % % % % % % %
\section{Geodesic Finite Elements}\label{sec:ch2}
We now introduce one specific way to construct discrete approximation spaces $S_{h}^{m}$.
Geodesic finite elements have been introduced as a means to interpolate data in a Riemannian manifold \cite{grohs}, 
and to solve partial differential equations from liquid crystal theory and Cosserat mechanics \cite{sander12, sander13}.
\subsection{Definition and General Properties}

Let $G$ be a conforming grid of width $h$ and order $m\in \NN$ with reference element $T\in \RR^{d}$.
For a set of Lagrange nodes $a_{i}\in T$, $i=1,\ldots,l$, let $\lambda_{i} : T \to \RR$, $i=1,\ldots,l$, denote the corresponding Lagrange polynomials of order $m$, i.e.,
\begin{align*}
\lambda_{i}(a_{j})=\delta_{ij}\qquad \forall 1\leq i,j,\leq l,
\qquad \textrm{and}\quad \sum_{i=1}^{l}\lambda_{i}(x)=1\qquad \forall x\in T.
\end{align*}
The following generalization of Lagrangian interpolation was given and motivated in \cite{sander13}.

\begin{definition}
	Let $v_i \in M$, $i=1,\dots,l$ be values at the
	corresponding Lagrange nodes. We call 
	\begin{align}\label{eq:geodesic_interpolation}
	\Upsilon & \; : \; M^l \times T \to M, \qquad 
	\Upsilon(v_1,\dots,v_l;x) = \argmin_{q \in M} \sum_{i=1}^l \lambda_i(x) d(v_i,q)^2
	\end{align}
	$m$-th order geodesic interpolation on $M$.
	The set of all such functions will be denoted by $P_{m}(T,M)$.
\end{definition}

It is easy to verify that this definition reduces to $m$-th order Lagrangian interpolation
if $M=\RR^{n}$ and if $d(\cdot,\cdot)$ denotes the Euclidean distance.
We have chosen the letter $P$ to denote these functions to point out that they form a generalization of interpolatory polynomials to nonlinear (even metric) codomains.

For manifolds with either negative sectional curvature, or certain restrictions on the curvature and the $v_{i}$, 
well-posedness of the definition for $m=1$ is a classic result by Karcher \cite{Karcher1977}.
For $m \ge 2$, where the $\lambda_i$ can become negative, well-posedness has been proven in~\cite{sander13}.
The following slightly weaker well-posedness result has a much simpler proof.
\begin{lemma}[\cite{diss}]\label{L:wellposed}
	For $i=1,\ldots,l$ let $v_{i}\in M$ with $d(v_{i},v_{1})\leq \rho$ for all $i=1,\ldots,l$.
	Then for all $x\in T$, the set $\Upsilon(v_1,\dots,v_l;x)$ is non-empty, and there exists a constant $\Cl{c:wellposed}\leq 6l\max_{i}\|\lambda_{i}\|_{\infty}$, such that for each $x\in T$ the inclusion
	$\Upsilon(v_1,\dots,v_l;x) \subset B_{\Cr{c:wellposed}\rho}(v_{1})$ holds.
	\\
	If $\rho$ is small enough depending on the curvature of $M$, the solution $v_{I}(x)$ to the minimization problem defining $\Upsilon$ is unique, and the map $v_{I}:T\to M$ is smooth.
\end{lemma}
The existence follows by compactness of $\overline{B}_{\Cr{c:wellposed}\rho}(v_{1})$. Uniqueness and smoothness use the implicit function theorem.

Since the values of $\Upsilon$ are defined as solutions of a minimization problem, we can also characterize
them by the corresponding first-order optimality condition (see, for instance, \cite{Karcher1977}).
\begin{lemma}\label{L:varInterpol}
	The minimizer $q^* \colonequals \Upsilon(v_1,\dots,v_l;x)$ is (locally uniquely) characterized
	by the first-order condition
	\begin{equation}\label{eq:firstorderinterpol}
	\sum_{i=1}^l\lambda_i(x)\log_{q^*} v_i
	= 0  \; \in \; T_{q^*}M.
	\end{equation}
\end{lemma}
As in the linear setting, global finite elements are defined as continuous functions that are interpolants on each grid element.
\begin{definition}[Geodesic finite elements]\label{def:gfe}
	Let $M$ be a Riemannian manifold and $G$ a grid for a
	$d$-dimensional domain $ \Omega$, $d\geq 1$.  A geodesic finite element function
	is a continuous function $v_h : \Omega \to M$ such that for each element
	$T$ of $G$, $v_h|_T \in P_{m}(T,M)$.
	The space of all such functions will be called $\gfe$.
\end{definition}
\begin{remark}
	Geodesic finite elements can be used to interpolate continuous functions $v\in C(\Omega,M)$.
	Indeed, if $G$ is a grid for $\Omega$ of width $h$ and order $m$ (cf. Definition~\ref{def:widthh}), 
	by continuity of $v$ there exists an $h_{0}$ such that for all $h\leq h_{0}$ 
	the interpolation nodes for each element are contained in a ball of radius $\rho$, 
	and thus the nodal interpolant $v_{I}\in \gfe$ is well-defined.
	
	Note furthermore that we can always control the diameter $\diam(v_{I}(T_{h}))$ of 
	the image of an element $T_{h}$ of $G$ under $v_{I}$ by choosing $h_{0}$ small enough.
\end{remark}
Already in \cite{sander12} it was observed, that geodesic finite elements are $H^1$-functions in the sense of Definition~\ref{def:NashSobolev}.

\subsection{Error Estimates for Geodesic Finite Elements}
In \cite{grohsSanderH} a priori $W^{1,2}$-error estimates were shown for geodesic finite elements.
An abstract version of this result is Theorem~\ref{T:H1err}.

We now show a priori error bounds in the $L^2$-sense. For this, we will show that the assumptions of Theorem~\ref{T:L2err} are fulfilled. This leads to the following theorem.
\begin{theorem}\label{T:gfe}
	Let $2(m+1)>d$, $m\geq 1$.
	Choose $q$ as in Lemma~\ref{L:chooseq} and set $H = W^{1,q}_{\phi}(\Omega,M)$.
	Let $\Energy:H\to \RR$ be a predominantly quadratic functional in the sense of 	Definition~\ref{def:almostLinear}.
	Assume that $u\in  H\cap W^{m+1,2}(\Omega,M)$ is a minimizer of $\Energy:H\to \RR$ w.r.t.\ variations along geodesic homotopies in $H$, and that $\Energy$ is elliptic along geodesic homotopies starting in $u$.
	
	Let $G$ be a conforming grid of width $h$ and order $m$, for that each $F^{-1}_{h}:T\to T_{h}$ scales with order $2$. Set $V_{h}\colonequals H\cap \gfe$.
	Assume that the boundary data $\phi_{|\partial \Omega}$ is such that $V_{h}$ is not empty.
	
	Let $u_{I}\in \gfe$ be the geodesic interpolant of $u$, and let $K$ be a constant such that
	\begin{align}
	K\geq \Cr{c:m+1Deriv}\theta_{1,q,\Omega}(u).
	\end{align}
	Assume that $h$ is small enough such that $u_{I}\in H_{K,L}$, where $H_{K,L}$ is defined as in Lemma~\ref{L:cea}.
	
	Suppose that the deformation problem is $H^{2}$-regular, i.e., that the solution $W$ of \eqref{eq:deformedContVec} fulfills~\eqref{eq:H2reg}.
	
	Then $K$ can be chosen such that the discrete minimizer
	\begin{align*}
	u_{h}=\argmin_{v_{h}\in V_{h}\cap H_{K,L}} \Energy(v_{h})
	\end{align*}
	is a local minimizer in $V_{h}$ and fulfills the a priori error estimate
	\begin{align}\label{eq:gfeTheorem}
	d_{L^{2}}(u,u_{h}) + h\;D_{1,2}(u,u_{h})\leq C h^{m+1}\left(\theta_{m+1,2,\Omega}(u)+\theta^{2}_{m+1,2,\Omega}(u)\right).
	\end{align}
	If the error is measured in an isometric embedding $\iota:M\to \RR^{N}$, we have
	\begin{align}
	\|\iota\circ u - \iota\circ u_{h}\|_{0,2,\Omega} +
	h\; \|\iota\circ u - \iota\circ u_{h}\|_{1,2,\Omega}\leq C\;h^{m+1}\left(\|\iota\circ u\|_{m+1,p,\Omega}+\|\iota\circ u\|_{m+1,p,\Omega}^{2(m+1)}\right).
	\end{align}
\end{theorem}
Theorem~\ref{T:gfe} follows if we can show that geodesic finite elements fulfill the assumptions of Theorems~\ref{T:H1err} and~\ref{T:L2err}. It is enough to show that the four Conditions~\ref{cond:1},~\ref{cond:2b},~\ref{cond:3}, and~\ref{cond:4} are fulfilled.
In particular, we will show that the required estimates hold for geodesic interpolants.
\subsubsection{Condition~\ref{cond:1}: Interpolation Error Bounds for GFEs}
Let $kp>d$, $m\geq k-1$, and $u\in W^{k,p}(\Omega,M)$ with $u(T_{h})\subset B_{\rho}\subset M$ for all elements $T_{h}\in G$, where $\rho$ is small enough such that the geodesic interpolant $u_{I}\in \gfe$ is well-defined.

We begin by considering the a priori bound \eqref{eq:cond1b}.
Note that $(m+1)$-th order derivatives of geodesic finite elements do not vanish in general, as it is the case with Euclidean finite elements.
\begin{proposition}\label{prop:m+1Deriv}
	Let $M$ possess a $C^{k}$-atlas with $kp>d$, and let $m\geq k-1$.
	Then there exists a constant $C$ such that for all $u\in 
	W^{k,p}(T,M)$ with $u(T)\subset B_{\rho}\subset M$
	\begin{align*}
	\dot{\theta}_{k,p,T}(u_{I})\leq C\;\dot{\theta}_{k,p,T}(u)
	\end{align*}
	holds for the interpolation $u_{I}\in P_{m}(T,M)$ of $u$.
\end{proposition}
The proposition can be shown by contradiction. Details can be found in \cite{diss}.

A proof that GFE spaces fulfill the interpolation estimate \eqref{eq:cond1a} has appeared in \cite{grohsSanderH}. We state a slightly different version from \cite{diss} that has better constants, in particular allowing for geometric estimates for first order finite elements.
\begin{lemma}\label{L:BHLInterpol}
	Let $kp>d$, $m\geq k-1$, and $u\in W^{k,p}(T,M)$ with $u(T)\subset B_{\rho}\subset M$.
	Let $u_{I}\in P_{m}(T,M)$ denote the geodesic interpolation of $u$.
	Then there exists a constant $\Cl{L:BHLInterpol}$ such that
	\begin{align}
	d_{L^{p}(T,M)}(u,u_{I}) &\leq \Cr{L:BHLInterpol}C_{1,u}(T)\;\dot{\theta}_{k,p,T}(u), \label{eq:BHLInterpolLp}
	\end{align}
	and
	\begin{align}
	D_{1,p,T}(u,u_{I}) &\leq \Cr{L:BHLInterpol}\left(C^{p}_{1,u}(T) +C^{p}_{2,u}(T)\right)^{\frac{1}{p}}\;\dot{\theta}_{k,p,T}(u),\label{eq:BHLInterpolW1p}
	\end{align}
	where
	\begin{align*}
	C_{1,u}(T)&\colonequals 
	\sup_{1\leq j\leq k} \sup_{\genfrac{}{}{0pt}{1}{p\in u_{I}(T)}{q\in u(T)}}
	\|d^{j}\log_{p}{q}\|,\qquad  
	C_{2,u}(T)\colonequals \sup_{1\leq j\leq k-1} \sup_{\genfrac{}{}{0pt}{1}{p\in u_{I}(T)}{q\in u(T)}}\|d_{2}d^{j}\log_{p}{q}\| ,
	\end{align*}
	and the constant $ \Cr{L:BHLInterpol}$ depends on the shape functions $\lambda_{i}$, but is independent of $u$ and $M$.
\end{lemma}
The proof is based on the first-order optimality condition~\eqref{eq:firstorderinterpol} and uses a Taylor expansion of the vector field $\log_{u_{I}(x)}u(y)$. Proposition~\ref{prop:m+1Deriv} is used to estimate the smoothness descriptor of $u_{I}$.
\begin{corollary}
	Geodesic finite elements fulfill Condition~\eqref{cond:1}.
\end{corollary}
Estimates~\eqref{eq:cond1b} and~\eqref{eq:cond1a} follow directly from Proposition~\ref{prop:m+1Deriv} and Lemma~\ref{L:BHLInterpol} using the scaling properties of $D_{1,2}$ and the smoothness descriptor (see Section~\ref{sec:ch1}).
\subsubsection{Condition~\ref{cond:2b}: Interpolation Error Estimates for Variations of GFEs}
First note that on each element $T_{h}$ geodesic finite elements are smooth, i.e., $P_{m}(T_{h},M)\subset C^{\infty}(T_{h},M)\cap C(\overline{T}_{h},M)$. Therefore we have $P_{m}(T_{h},M)\subset  W^{2,b}(T_{h},M)$ for all $b$, in particular as specified in Condition~\ref{cond:2b} for
\begin{align*}
b\colonequals
\left\{
\begin{array}{ll}
2 &\quad\textrm{for}\ 2k>d,\\
3 &\quad\textrm{for}\ 2k=d,\\
\frac{d}{k} &\quad\textrm{for}\ 2k<d.
\end{array}
\right. 
\end{align*}

The variation of geodesic interpolants through geodesic interpolants induces a natural definition of the interpolation of vector fields along a geodesic interpolant.
\begin{definition}\label{def:vfieldinterpol}
	Let $\polyp\in P_{m}(T,M)$, and let $V^{i}\in T_{\polyp(a_{i})}M$, $i=1,\ldots,l$, be vectors given at the $l$ Lagrange nodes.
	Set $v_{i}(t)\colonequals \exp_{\polyp(a_{i})}(tV^{i})$ for $i=1,\ldots,l$.
	The interpolating vector field $V_{I}$ along $\polyp$ is then defined by
	\begin{align*}
	V_{I}(x)\colonequals \frac{d}{dt}\bigg|_{t=0}\Upsilon(v_{1}(t),\ldots,v_{l}(t);x).
	\end{align*}
	We denote the space of all interpolating vector fields along $\polyp$ by $IV(T,\polyp^{-1}TM)$.
\end{definition}
Note that the interpolating vector fields are generalized Jacobi fields in the 
same sense as geodesic finite elements are generalized geodesics.
\begin{remark}
	The interpolation of vector fields along a discrete function $\polyp$ 
	is well defined as long as geodesic interpolation of the points $\exp_{\polyp(a_{i})}(tV^{i})$ is well defined and smooth for small $t$.
	Smoothness follows by smoothness of the geodesic finite element 
	interpolation \cite{sander12}.
	Indeed, we can differentiate \eqref{eq:firstorderinterpol} for 
	$\Upsilon(v_{1}(t),\ldots,v_{l}(t);x)$ with respect to $t$ and obtain
	\begin{align}\label{eq:VIimplicit}
	V_{I}(x)=\sum_{i=1}^{l}\lambda_{i}(x)\left(Id+ 
	d_{2}\log_{\polyp(x)}\polyp_{i}\right)(V_{I}(x)) + 
	\sum_{i=1}^{l}\lambda_{i}(x)d\log_{\polyp(x)}\polyp_{i}(V_{i})
	\end{align}
	as an implicit formula for $V_{I}$. For $\diam(\polyp(T))\leq 
	\frac{1}{\sqrt{|R|_{\infty}\|\sum_{i}|\lambda_{i}|\|_{\infty}}}$, this yields in 
	particular
	\begin{align}\label{eq:absBoundVI}
	|V_{I}(x)|\leq C \max_{i}|V_{i}|.
	\end{align}
	Note further that for a constant function $\polyp$, vector field interpolation 
	corresponds to polynomial interpolation in $\RR^{n}$.
\end{remark}

\begin{remark}\label{R:geodesicVfields}
	Geodesic vector field interpolation is defined by variation of geodesic 
	interpolants. However, we can also see it as a variational form of geodesic
	interpolation on $TM$ 
	with respect to the pseudo-Riemannian metric $\;^{h}g$ defined by the horizonal lift 
	(see~\ref{sec:metricsTM}).
	By this we mean that if $(u_{i},V^{i})$ denote values in $TM$, $u_{I}$ the geodesic 
	interpolation of $u_{i}$ in $M$, and $V_{I}$ the interpolation of the $V^{i}$ 
	in the sense of Definition~\ref{def:vfieldinterpol}, we have
	\begin{align}\label{eq:firstorderOptV}
	\sum_{i=1}^{l}\lambda_{i}(x)
	\left[\log_{u_{I}(x)}u_{i}, d\log_{u_{I}(x)}u_{i} (V^{i}) +\;d_{2}\log_{u_{I}(x)}u_{i} (V_{I}(x))
	\right]
	&=[0,0]\;\in (T_{u_{I}(x)}M)^{2}.
	\end{align}
	Note that we do not obtain a minimization formulation of geodesic vector field interpolation as $\;^{h}g$ is only a pseudo-metric.
\end{remark}
We begin by considering the a priori bound~\eqref{eq:dgeq4bound}.
\begin{proposition}\label{prop:vecInterpolBd}
	Let $u\in W^{2,b}(T,M)\cap C(\overline{T},M)$, $b$ as above, with $u(T)\subset B_{\rho}\subset M$, $\rho$ small enough, and let $V\in W^{2,2}(T,u^{-1}TM)\cap C(\overline{T},u^{-1}TM)$ with $|V(x)|_{g(u(x))}\leq R$ for all $x\in T$. Let $(u_{I},V_{I})$ denote the $m$-th order GFE interpolant of $(u,V)$ for some $m\geq 1$.
	Then there exists a constant $C$ such that
	\begin{align}\label{eq:vecInterpolBd}
	\Theta_{2,2,T}(V_{I})\leq C \Theta_{2,2,T}(V).
	\end{align}
\end{proposition}
\begin{proof}
	Assume the converse, i.e., for all $K>0$ assume that there exist $u^{K}\in W^{2,b}(T,M)\cap C(\overline{T},M)$
	and $V_{K}\in W^{2,2}(T,u^{-1}TM)\cap C(\overline{T},u^{-1}TM)$ with $u_{i}(K)=u^{K}(a_{i})\in B_{\rho}$ and $V^{i}(K)=V_{K}(a_{i})\in B_{R}\subset T_{u_{i}(K)}M$
	such that
	\begin{align*}
	\Theta_{2,2,T}(V^{K})<\frac{1}{K}\; \Theta_{2,2,T}(V^{K}_{I}).
	\end{align*}
	Then w.l.o.g. $u_{i}(K)\to u_{i}^{*}\in \overline{B}_{\rho}$ for $K\to \infty$.
	As parallel transport preserves the length, we also have 
	\begin{align*}
	|\pi_{u_{i}(K)\mapsto u_{i}^{*}}V^{i}(K)|_{g(u_{i}^{*})}=|V^{i}(K)|_{g(u_{i}(K))}\leq R,
	\end{align*}
	and thus w.l.o.g.
	\begin{align*}
	\pi_{u_{i}(K)\mapsto u_{i}^{*}}V^{i}(K)\to V^{i}_{*}\in \overline{B}_{R}\subset T_{u_{i}^{*}}M.
	\end{align*}
	We define $(u_{I}^{\star}, V_{I}^{\star})$ as the first-order interpolation of $(u_{i}^{*},V^{i}_{\star})$.
	By smoothness of the interpolation operator, we have
	\begin{align*}
	\Theta_{2,2,T}(V^{K}_{I})\leq C\; \Theta_{2,2,T}(V^{\star}_{I}),
	\end{align*}
	and hence
	\begin{align*}
	\Theta_{2,2,T}(V^{K})<\frac{C}{K}\; \Theta_{2,2,T}(V^{\star}_{I})\to 0 \ \textrm{for}\ K\to\infty.
	\end{align*}
	In particular this implies
	\begin{align*}
	\dot\theta_{2,b,T}(u^{K})\|V_{K}\|_{L^{a}}\to 0 \ \textrm{for}\ K\to\infty.
	\end{align*}
	If $\|V_{K}\|_{L^{a}}\to 0$, the continuity of the $V^{K}$ implies $V^{K}(a_{i})\to 0$ and hence $V_{I}^{\star}\equiv0$, which is a contradiction.
	If $\dot\theta_{2,b,T}(u^{K})\to 0$, then $u^{K}(x)\to p\in M$ for all $x\in T$ as $b>\frac{d}{2}$. Thus $u_{I}^{\star}\equiv p\in M$.
	This implies that $V_{I}^{\star}=\sum_{i}\lambda_{i}V_{i}^{*}$ is an $m$-th order polynomial into the linear space $T_{p}M$.
	As $V_{I}^{\star}=\lim_{K\to \infty}\left(\pi_{u^{K}(x) \mapsto p} V_{K}(x)\right)_{I}$, we have
	\begin{align*}
	\dot\Theta_{2,2,T}(V^{\star}_{I})=\|\nabla^{2} V^{\star}_{I}\|_{L^{2}}\leq C \|\nabla^{2}(\pi_{u^{K} \mapsto p} V_{K})_{I}\|_{L^{2}}\leq C\; \Theta_{1,1,T}(V_{K}),
	\end{align*}
	and thus a contradiction for large $K$.
\end{proof}
In \cite{diss} a similar proof to that of Lemma~\ref{L:BHLInterpol} in \cite{grohsSanderH} leads to the following.
\begin{lemma}\label{L:BHLInterpolVfield}
	Let $m\geq k-1$ and $(u,V)\in W^{k,p}\cap C(\overline{T},TM)$ with $u(T)\subset B_{\rho}\subset M$, and $u\in W^{k,b}(T,M)$ with $b$ as in Definition~\ref{def:nonlinSmoothVfield}.
	Let $(u_{I},V_{I})$ denote the geodesic interpolant of $(u,V)$.
	Then there exists a constant $\Cl{c:BHLInterpolVfield}$ such that
	\begin{multline}
	\left\|d\log_{u_{I}(x)}u(x)(V(x))+d_{2}\log_{u_{I}(x)}u(x)(V_{I}(x))\right\|_{W^{1,p}(T,u_{I}^{-1}TM)} \\
	\leq \Cr{c:BHLInterpolVfield}\;C_{u,k+1}(T) \left(\dot \Theta_{k,p,\Omega}(V_{I})+\dot \Theta_{k,p,\Omega}(V) \right),
	\end{multline}
	where
	\begin{align*}
	C_{u,k+1}(T)\colonequals \sup_{i+j\leq k+1} \sup_{\genfrac{}{}{0pt}{1}{p\in u_{I}(T)}{q\in u(T)}}\|d_{2}^{i}d^{j}\log_{p}{q}\|.
	\end{align*}
\end{lemma}
For dimensions $d<4$ the interpolation error estimate for vector fields implies \eqref{eq:vecInterpolCond} in Condition~\ref{cond:2b}.
\begin{corollary}
	For $d<4$, geodesic finite elements fulfill Condition~\ref{cond:2b}.
\end{corollary}
As for functions, the estimates on the elements $T_{h}$ are obtained via the scaling properties of the Sobolev norm and the smoothness descriptor (see Lemma~\ref{L:scalingSDVec}).

The restriction on the dimension comes from the fact that geodesic interpolation is only defined for continuous vector fields.
If we have a function $u\in C(\Omega,M)$, and a vector field $V\in W^{k,p}(\Omega,u^{-1}TM)$ along $u$ that is of a different smoothness than the base function $u$, we say that $(u,V)\in C(\Omega,M)\times_{\pi} W^{k,p}(\Omega,TM)$. This notation extends to other specifications of smoothness as well.

In order to generalize to vector fields that may have discontinuities, we use the interpolation of mollified vector fields.
Mollification, as well as other techniques to deal with discontinuities like, e.g., Cl\'ement-interpolation, always introduce a kind of non-locality, i.e., we cannot expect the estimates~\eqref{eq:dgeq4bound} and~\eqref{eq:vecInterpolCond} to hold locally on each element $T_{h}$ of the grid anymore. Thus, we weaken Condition~\ref{cond:2b} to the following:
\begin{condition}\label{cond:2c}
	For any element $T_{h}$ from the grid $G$, let $S_{h}^{m}(T_{h},M)\subset W^{2,b}(T_{h},M)$ for $b$ as in Definition~\ref{def:nonlinSmoothVfield} with $p=k=2$.
	Let 
	\begin{align*}
	\overline{\omega}_{h}\colonequals \bigcup\left\{\overline{T}\in G\;:\;\overline{T}_{h}\cap \overline{T}\neq\emptyset\right\}
	\end{align*}
	be the macro-element determined by $T_{h}$.
	Given $(u_{h},V)\in S_{h}^{m}(\omega_{h},M)\times_{\pi} W^{2,2}(\omega_{h},TM)$, let there exist a family of maps $v_{h}(t)\in S_{h}^{m}(\omega_{h},M)$ with $v_{h}(0)=u_{h}$ and and constants $\Cl{c:cond2c}, \Cl{c:cond2c2}$ such that
	for $V_{I}=\frac{d}{dt}v_{h}(0)$ the estimates
	\begin{align}\label{eq:dgeq4boundc}
	\Theta_{2,2,T_{h}}(V_{I})\leq \Cr{c:cond2c}\; \Theta_{2,2,\omega_{h}}(V)
	\end{align}
	and
	\begin{align}\label{eq:vecInterpolCondc}
	\| V - V_{I}\|_{W^{1,2}(T_{h},u_{h}^{-1}TM)} &\leq \Cl{c:cond2c2} h \;\left(
	\Theta_{2,2,\omega_{h}}(V) + 	\Theta_{2,2,\omega_{h}}(V_{I})\right)
	\end{align}
	hold for the grid-dependent smoothness descriptor $\Theta_{2,2,\omega_{h}}$.
\end{condition}
As the shape regularity of the grid implies bounded overlap of the $\omega_{h}$, we can view Condition~\ref{cond:2c} as tantamount to Condition~\ref{cond:2b}.

Note that we are still restricted to continuous functions, as this is not only needed for point values but also for using coordinates, and the logarithm along a function. Even for smoothing a map $u$, the radius of the image under $u$ needs to be controllable by the radius of the preimage \cite{Karcher1977} (although $u$ might be undefined on a set of measure 0).

It is well-known how to define the mollification of a function $u\in C(\Omega,M)$ \cite{Karcher1977}, namely, for a small smoothing parameter $\epsilon>0$ we define the mollification operator $R_{\epsilon}:C(\Omega,M)\to C^{\infty}(\Omega_{\epsilon},M)$ by
\begin{align}\label{eq:mineps}
R_{\epsilon}u(x)\colonequals u_{\epsilon}(x)\colonequals \argmin_{q\in M}\frac{1}{2}\int_{B_{\epsilon}(x)}\phi_{\epsilon}(|x-y|)d^{2}(u(y),q)\;dy
\end{align}
at every $x\in \Omega_{\epsilon}\colonequals \left\{ x\in \Omega\;:\;d_{\RR^{d}}(x,\partial\Omega)>\epsilon\right\}$, where $\phi_{\epsilon}$ denotes a standard mollification kernel (cf., e.g., \cite{Evans}).
This minimization is well-posed for continuous functions $u$ and $\epsilon$ small enough depending on the injectivity radius of $M$ at $u(x)$.

Making use of mollification requires the extension of maps to a fixed larger domain $\tilde{\Omega}$, such that the closure of $\Omega$ is contained in $\tilde{\Omega}$. We then choose the mollification parameter $\epsilon$ small enough such that $\Omega\subset \tilde{\Omega}_{\epsilon}$.

Thus, given $u_{h}\in S_{h}(\Omega,M)$ and $V\in L^{1}(\Omega,u_{h}^{-1}TM)$, we first extend $u$ and $V$ to $\tilde{\Omega}$, and then set %for a small mollification parameter $\epsilon>0$
\begin{align}\label{eq:smoothedFamily}
I_{\epsilon h}\big(u_{h;V}(t)\big)\colonequals \left(R_{h\epsilon}\big(u_{h;V}(t)\big)\right)_{I},
\end{align}
where $u_{h;V}(t)\in L^{1}(\Omega,M)$ with $u_{h;V}(0)=u_{h}$ and $\dot u_{h;V}(0)=V$.
%In order to define $v_{h}$ on all of $T_{h}$ we need to extend $(u_{h},V)$ smoothly outside of $T_{h}$ to some $\hat T_{h}$ with $T_{h}\subset \hat T_{h\epsilon}$.
We define a discrete family of maps by first smoothing with a mollification parameter proportional to the discretization parameter $h$, and then interpolating. Note that the operator $I_{\epsilon h}:L^{1}(\Omega,M)\to S_{h}(\Omega,M)$ is only well-defined if the mollified function still fulfills the condition that its values at neighboring Lagrangian nodes are contained in a $\rho$-ball. Thus, $\epsilon$ has to be small enough to ensure this.

The operator $I_{\epsilon h}$ defines an interpolation operator on vector fields by
\begin{align*}
H_{\epsilon h}: S_{h}(\Omega,M)\times_{\pi} L^{1}(\Omega,TM)&\to S_{h}(\Omega,TM),\\
(u_{h},V)&\mapsto (I_{\epsilon h}u, V_{\epsilon I}),
\end{align*}
where $V_{\epsilon I}=\frac{d}{dt}|_{t=0}I_{\epsilon h}\left(u_{h;V}(t)\right)$ denotes the variational field of the family.
However, $V_{\epsilon I}$ is not defined along the original discrete function $u_{h}$ as $I_{\epsilon h}u_{h}\neq u_{h}$. To overcome this problem, we could introduce a transport of this vector field to $u_{h}$ and interpolate along this discrete function. This leads to the unfavorable property that the obtained interpolation operator for vector fields does not preserve discrete vector fields.

Borrowing ideas from \cite{arnold2006finite}, we instead restrict the operator $I_{\epsilon h}$ to the set of discrete maps in $K(\tilde{\rho})\subset S_{h}$ whose neighboring nodal values at the Lagrangian nodes $a_{i}$ are not further than $\tilde{\rho}$ apart, where $\tilde{\rho}$ is smaller than $\rho$. Note that we can identify this set with a compact set $\tilde{K}(\tilde{\rho})\subset M^{n_{l}}$, where $n_{l}$ denotes the number of Lagrange nodes on $G$.
With this identification $E_{\epsilon h}\colonequals I_{\epsilon h}|_{K(\tilde{\rho})}$ corresponds to an operator $\tilde E_{\epsilon h}:\tilde{K}(\tilde{\rho})\to  M^{n_{l}}$. We further introduce a dependence on $\epsilon$, i.e.
$\tilde E_{h}:[0,\delta]\times \tilde{K}(\tilde{\rho})\to  M^{n_{l}}$, $\tilde{E}_{h}(\epsilon,\cdot)\colonequals \tilde{E}_{\epsilon h}$. For any $\vec{v}\in \tilde{K}(\tilde{\rho})$, the differential $d\tilde{E}_{h}((0,\vec{v}))$ is the identity map and thus invertible. Thus, there exists a neighborhood $U_{\vec{v}}\subset \tilde{K}(\tilde{\rho})$ such that $\tilde{E}_{h}$ is invertible on $[0,\delta(\vec{v}))\times U_{\vec{v}}$. As $\tilde{K}(\tilde{\rho})$ is compact, we can use a finite covering and choose a minimal $\delta$ to obtain invertibility of $\tilde {E}_{h}$ on $[0,\delta_{\min})\times \tilde{K}(\tilde{\rho})$.

Thus $E_{\epsilon h}$ is invertible, and
we define a family of discrete maps by
\begin{align*}
v_{h}(t)\colonequals E_{\epsilon h}^{-1}\circ I_{\epsilon h}\big(u_{h;V}(t)\big).
\end{align*}
This is well-defined for $u_{h;V}(t)\in K(\tilde\rho)$, i.e., for functions $u_{h}$ such that values at neighboring Lagrange nodes are contained in an open $\tilde \rho$-ball, and $t$ small enough. Obviously, the property $v_{h}(0)=u_{h}$ is fulfilled.
Finally, we define the interpolation of $V$ by
\begin{align*}
V_{I}\colonequals \frac{d}{dt}\bigg|_{t=0}v_{h}(t).
\end{align*}
This defines an operator
\begin{align*}
F_{\epsilon h}: S_{h}(\Omega,M)\times_{\pi} L^{1}(\Omega,TM)&\to S_{h}(\Omega,TM),\\ (u_{h},V)&\mapsto (u_{h},V_{I}),
\end{align*}
that leaves discrete vector fields unchanged.

Assuming $(u_{h},V)\in S_{h}^{m}(\Omega,M)\times_{\pi} W^{2,2}(\Omega,TM)$, we now need to show~\eqref{eq:dgeq4boundc} and~\eqref{eq:vecInterpolCondc} for $V_{I}$ on each element $T_{h}$ with corresponding macroelement $\omega_{h}$.
As we have $F_{\epsilon h}\circ  H_{\epsilon h}= H_{\epsilon h}=H_{\epsilon h}\circ F_{\epsilon h}$, % , as we have
%\begin{align*}
%H_{\epsilon h}(F_{\epsilon h}(u_{h},V))
%&=H_{\epsilon h}(u_{h},  DE_{\epsilon h}^{-1}(I_{\epsilon h}u_{h})(V_{\epsilon I}))\\
%&=(I_{\epsilon h}u_{h}, DI_{\epsilon h}(u_{h})(DE_{\epsilon h}^{-1}(I_{\epsilon h}u_{h})(V_{\epsilon I}))\\
%&=(I_{\epsilon h}u_{h}, V_{\epsilon I}).
%\end{align*}
we will first prove the estimates for $V_{\epsilon I}$, and then use these relations to deduce them for $V_{I}$. As $V_{\epsilon I}$ is defined by subsequent mollification and interpolation, we will begin by considering the effect of mollification.

\begin{proposition}[Boundedness of mollification]\label{prop:mollibd}
	Let $T_{h}\in G$ with macro-element $\omega_{h}$. Let $\epsilon$ be small enough such that $T_{h}\subset \tilde{T}_{h;\epsilon}\subset \tilde{T}_{h}\subset \omega_{h}$.
	For $(u,V)\in C\cap W^{k,p_{k}}(\omega_{h},M)\times_{\pi} W^{k,p}(\omega_{h},TM)$, the mollification $V_{\epsilon}=\frac{d}{dt}\big|_{t=0}R_{\epsilon}\left(u_{V}(t)\right)$ fulfills
	\begin{align}\label{eq:fullbdd}
	\dot \Theta_{k,p,T_{h}}(V_{\epsilon})\leq C\; \dot \Theta_{k,p,\omega_{h}}(V)
	\end{align}
	if $\epsilon$ is small enough depending on the curvature of $M$.
\end{proposition}
\begin{proof}
	A variational characterization of $V_{\epsilon}$ is given by
	\begin{align}\label{eq:1stordervec}
	\int_{B_{\epsilon}(x)}\phi_{\epsilon}(|x-y|)
	\left(d\log_{u_{\epsilon}(x)}u(y) (V(y)) + d_{2}\log_{u_{\epsilon}(x)}u(y) (V_{\epsilon}(x))\right)
	\;dy=0\in T_{u_{\epsilon}(x)}M\;.
	\end{align}
	As $u$ is continuous, $u_{\epsilon}=R_{\epsilon}u$ converges to $u$ uniformly on compact subsets.
	Choosing $\epsilon$ small enough such that $d(u(x),u(y))+d(u_{\epsilon}(x),u(x))\leq \frac{1}{\sqrt{2\|\Rm\|_{\infty}}}$ for $|x-y|\leq \epsilon$, we can use Proposition~\ref{prop:logest} to estimate
	\begin{align*}
	|V_{\epsilon}(x)|&\leq 3 \int_{B_{\epsilon}(x)}\phi_{\epsilon}(|y-x|)\left|V(y)\right| \;dy.
	\end{align*}
	Thus, for $k=0$ the bound~\eqref{eq:fullbdd} is fulfilled.
	As the differential and mollification can be interchanged for functions, i.e., $\left(du^{\alpha}\right)_{\epsilon}=du_{\epsilon}^{\alpha}$, we can write for $k=1$
	\begin{align*}
	\dot\Theta_{1,p}^{p}(V_{\epsilon})
	%&=\|\nabla V_{\epsilon}\|_{L^{p}}^{p} + \|du_{\epsilon}\|_{L^{p_{0}}}^{p}\|V_{\epsilon}\|_{L^{r_{0}}^{p}\\
	&=\|\nabla V_{\epsilon}\|_{L^{p}}^{p} + \dot\Theta_{0,p_{0}}^{p}(du_{\epsilon})\dot\Theta^{p}_{0,r_{0}}(V_{\epsilon})\\
	&\leq \|\nabla V_{\epsilon}\|_{L^{p}}^{p} + C\; \dot\Theta_{0,p_{0}}^{p}(du)\dot\Theta^{p}_{0,r_{0}}(V).
	\end{align*}
	Since for vector fields $(\nabla V)_{\epsilon}\neq \nabla V_{\epsilon}$, we differentiate~\eqref{eq:1stordervec} to obtain
	\begin{multline*}
	-\int_{B_{\epsilon}(x)}\phi_{\epsilon}(|x-y|)
	d_{2}\log_{u_{\epsilon}(x)}u(y) \left(\nabla_{du_{\epsilon}}^{\alpha} V_{\epsilon}(x)\right)
	\;dy\\
	\begin{aligned}
	&= 		\int_{B_{\epsilon}(x)}\phi_{\epsilon}(|x-y|)
	\big( d_{2}d \log_{u_{\epsilon}(x)}u(y)\left(V(y),d^{\alpha}u_{\epsilon}(x)\right)\\
	&\qquad +d^2 \log_{u_{\epsilon}(x)}u(y)\left(V(y),d^{\alpha}u(y)\right)
	+  d_{2}^{2}\log_{u_{\epsilon}(x)}u(y)\left(V_{\epsilon}(x),d^{\alpha}u_{\epsilon}(x)\right)\\
	&\qquad +
	dd_{2} \log_{u_{\epsilon}(x)}u(y)\left(V_{\epsilon}(y),d^{\alpha}u(y)\right)
	+ d\log_{u_{\epsilon}(x)}u(y)\left(\nabla_{du}^{\alpha}V(y)\right)
	\big)\;dy.
	\end{aligned}
	\end{multline*}
	Proposition~\ref{prop:logest} then implies~\eqref{eq:fullbdd} for $k=1$.
	
	Induction over $k$ by further differentiating~\eqref{eq:1stordervec} yields the assertion. Note that we always obtain a description of the highest-order derivatives of $V_{\epsilon}$ in terms of the same derivatives of $V$, mixed lower-order derivatives of $V$, $V_{\epsilon}$, $u$, and $u_{\epsilon}$, and up to $k$-th order derivatives of the logarithm.
\end{proof}

\begin{proposition}
	Let $T_{h}\in G$ with macro-element $\omega_{h}$, and let $\epsilon$ be small enough. 
	Let $(u_{h},V)\in S_{h}^{m}(\omega_{h},M)\times_{\pi} W^{2,2}(\omega_{h},TM)$,
	such that values of $u_{h}$ at neighboring Lagrange nodes are contained in an open $\tilde{\rho}$-ball, where $\tilde\rho>0$ is small enough.
	Then $(u_{h},V_{I})=F_{\epsilon h}(u,V)$ fulfills \eqref{eq:dgeq4boundc}.
\end{proposition}

\begin{proof}
	First note that as $V_{\epsilon I}$ given by $(I_{\epsilon h}u_{h},V_{\epsilon I})=H_{\epsilon h}(u_{h},V)$ is defined  by subsequent mollification and interpolation, Propositions~\ref{prop:mollibd} and~\ref{prop:vecInterpolBd} imply
	\begin{align*}
	\Theta_{2,2,T_{h}}(V_{\epsilon I})\leq C\; \Theta_{2,2,\omega_{h}}(V).
	\end{align*}
	We further have $V_{I}= dE_{\epsilon h}^{-1}(I_{\epsilon h}u_{h})(V_{\epsilon I})$. Thus, we obtain by the chain rule
	\begin{align}\label{eq:boundByhatV}
	\Theta_{2,2,T_{h}}(V_{I})\leq C \Theta_{2,2,T_{h}}(V_{\epsilon I}),
	\end{align}
	where the constant depends on the operator norm $\|dE_{\epsilon h}^{-1}\|$.
	As $E_{\epsilon h}^{-1}$ is defined via the inverse function theorem, it is bounded, as $E_{\epsilon h}$ itself is bounded.
\end{proof}
In order to prove~\eqref{eq:vecInterpolCond}, we define on vector fields that are in 
$C\cap W^{2,b}(T,M)\times_{\pi} W^{2,2}(T,TM)$ on each element $T$
the mapping $D_{T}$ by
\begin{align*}
D_{T}((u,U),(v,V))\colonequals\left\| d\log_{u(x)}v(x)(V(x))+d_{2}\log_{u(x)}v(x)(U(x))\right\|_{W^{1,2}(T,u^{-1}TM)}.
\end{align*}
Note that $D_{T}$ fulfills an approximate triangle inequality of the form
\begin{align*}
D_{T}\left((u,U),(v,V)\right)&\leq C\big(D_{T}\left((u,U),(w,W)\right) + D_{T}\left((w,W),(v,V)\right)\big)\\
&\qquad  + C(d(u,w)+d(w,v))\big(\Theta_{2,2,T}(U)+ \Theta_{2,2,T}(V)+\Theta_{2,2,T}(W)\big).
\end{align*}
It can be proven by a technical calculation using Proposition~\ref{prop:logest} to estimate derivatives of $\log$.

This triangle inequality again allows us to consider the effect of mollification separately.
\begin{proposition}[Error estimate for mollification]\label{prop:smoothedErrorEst}
	Let $T_{h}\in G$ with macro-element $\omega_{h}$,
	let $(u_{h},V)\in S_{h}^{m}(\omega_{h},M)\times_{\pi} W^{2,2}(\omega_{h},TM)$, and $(u_{h\epsilon},V_{\epsilon})=R_{h\epsilon}(u_{h},V)$.
	Then, if $\epsilon$ is small enough,
	\begin{align}\label{eq:smoothedErrorEst}
	D_{T_{h}}((u_{h\epsilon},V_{\epsilon}),(u_{h},V))
	\leq C\;h\epsilon  \Theta_{2,2,\omega_{h}}(V).
	\end{align}
\end{proposition}
\begin{proof}
	The assertion follows from the variational characterization of $V_{\epsilon}$~\eqref{eq:1stordervec}, Proposition~\ref{prop:logest}, and Proposition~\ref{prop:mollibd}.
\end{proof}
\begin{proposition}\label{prop:cond2discVfields}
	Let $(u_{h},V)\in S_{h}^{m}(\omega_{h},M)\times_{\pi} W^{2,2}(\omega_{h},TM)$, such that values of $u_{h}$ at neighboring Lagrange nodes are contained in an open $\tilde{\rho}$-ball, where $\tilde\rho>0$ is small enough. 
	There exist $\epsilon$ and $h$ such that for $(u_{h},V_{I})=F_{\epsilon h}(u_{h},V)$
	\begin{align*}
	\|V-V_{I}\|_{W^{1,2}(T_{h},u_{h}^{-1}TM)}\leq C\;h\left(\Theta_{2,2,\omega_{h}}(V)+\Theta_{2,2,\omega_{h}}(V_{I})\right).
	\end{align*}
\end{proposition}

\begin{proof}
	For any $(v_{h},W)\in  S_{h}^{m}(\omega_{h},M)\times_{\pi} W^{2,2}(\omega_{h},TM)$
	the approximate triangle inequality, the definition of $H_{\epsilon h}$ as subsequent mollification and interpolation, Propositions~\ref{prop:smoothedErrorEst}, ~\ref{prop:mollibd}, and~\ref{prop:vecInterpolBd}, and Lemma~\ref{L:BHLInterpolVfield} imply
	\begin{align*}
	D_{T_{h}}((v_{h},W),H_{\epsilon h}(v_{h},W))
	&\leq C\;h\Theta_{2,2,\omega_{h}}(W).
	\end{align*}
	Further, the approximate triangle inequality in conjunction with $H_{\epsilon h}\circ F_{\epsilon h}=H_{\epsilon h}$ yields
	\begin{align*}
	\|V-V_{I}\|_{W^{1,2}(T_{h},u_{h}^{-1}TM)}&=D_{T_{h}}((u_{h},V),F_{\epsilon h}(u_{h},V))\\
	&\leq C\left[D_{T_{h}}((u_{h},V),H_{\epsilon h}(u_{h},V))+D_{T_{h}}(H_{\epsilon h}(u_{h},V_{I}),(u_{h},V_{I}))\right]\\
	&\qquad  + C\;d(u_{h},I_{h\epsilon}u_{h})\left[\Theta_{2,2,T_{h}}(V) + \Theta_{2,2,T_{h}}(V_{\epsilon I}) + \Theta_{2,2,T_{h}}(V_{I})\right]\\
	&\leq C\;h \left(\Theta_{2,2,\omega_{h}}(V) + \Theta_{2,2,\omega_{h}}(V_{I})\right).
	\end{align*}
	Propositions~\ref{prop:mollibd} and~\ref{prop:vecInterpolBd} yield the assertion.
\end{proof}

\begin{corollary}
	Geodesic finite elements fulfill Condition~\ref{cond:2c} for all $d\in \NN$.
\end{corollary}

\subsubsection{Condition~\ref{cond:3}: Inverse Estimates for GFEs}
For Condition~\ref{cond:3} we require the grid $G$ to be such that $F^{-1}_{h}:T\to T_{h}$ scales with order $2$ for all elements $T_{h}$ of $G$. We need to show for $v_{h}\in P_{m}(T_{h},M)$ with $v(T_{h})\subset B_{\rho}$ and $p,q\in[1,\infty]$ the inverse estimate~\eqref{eq:inverse2p1qScaled}.
As the homogeneous smoothness descriptor $\dot\theta_{k,p,T_{h}}$ consists of mixed-order derivatives, 
inverse estimates are more complicated than those for standard Sobolev half-norms.
An exception are the first-order smoothness descriptors.
\begin{proposition}\label{prop:1derivPm}
	Let $v\in P_{m}(T,M)$ with $v(T)\subset B_{\rho}$, $\rho\leq 
	\frac{1}{\sqrt{|R|_{\infty}\|\sum_{i}|\lambda_{i}|\|_{\infty}}}$.
	Let $p,q\in[1,\infty]$.
	Then 
	such that
	\begin{align}
	\dot{\theta}_{1,p,T}(v)&\leq C\;\dot{\theta}_{0,q,T;Q}(v), \label{eq:1derivPmRef}\\
	\dot{\theta}_{1,p,T}(v)&\leq C\;\dot{\theta}_{1,q,T}(v) \label{eq:inverse1p1q},
	\end{align}
	where $Q$ is the $v_{i}$ such that $\dot{\theta}_{0,p,T;v_{i}}(v)$ is maximal.
	
	After rescaling by a map $F:T_{h}\to T$ of order $1$, we have for $v_{h}\in P_{m}(T_{h},M)$
	\begin{align}
	\dot{\theta}_{1,p,T_{h}}(v_{h})&\leq C\;h^{-1 - d\max\left\{0,\frac{1}{q}-\frac{1}{p}\right\}}\;\dot{\theta}_{0,q,T_{h};Q}(v_{h}), \label{eq:1derivPmLoc} \\
	\dot{\theta}_{1,p,T_{h}}(v_{h})&\leq C\;h^{-d\max\left\{0,\frac{1}{q}-\frac{1}{p}\right\}}\;\dot{\theta}_{1,q,T_{h}}(v_{h}). \label{eq:inverse1p1qScaled}
	\end{align}
\end{proposition}
The proof follows from Lemma~\ref{L:wellposed}, norm equivalence in $\RR^{l}$, where $l$ is the number of degrees of freedom, and the differentiation of~\eqref{eq:firstorderinterpol}.
For higher-order smoothness descriptors, we obtain the following.
\begin{proposition}\label{prop:inverseEst21}
	Let $v\in P_{m}(T,M)$ with $v(T)\subset B_{\rho}$, $\rho\leq 
	\frac{1}{\sqrt{|R|_{\infty}\|\sum_{i}|\lambda_{i}|\|_{\infty}}}$.
	Let $p,q\in[1,\infty]$.
	Then 
	\begin{align}
	\dot{\theta}_{2,p,T}(v)&\leq \dot{\theta}^{2}_{1,2p,T}(v) + C\;\dot{\theta}_{1,q,T}(v).\label{eq:inverse2p1q}
	\end{align}
	As $F^{-1}_{h}:T\to T_{h}$ scales with order $2$, we have after rescaling for $v_{h}\in P_{m}(T_{h},M)$
	\begin{align}\label{eq:inverse2p1qScaledConcrete}
	\dot{\theta}_{2,p,T_{h}}(v_{h})\leq C\;\dot{\theta}^{2}_{1,2p,T_{h}}(v_{h}) + C\;h^{-1-d\max\left\{0,\frac{1}{q}-\frac{1}{p}\right\}}\;\dot{\theta}_{1,q,T_{h}}(v_{h}).
	\end{align}
\end{proposition}
The proof follows from differentiation of the first-order condition~\eqref{eq:firstorderinterpol} and using the scaling properties of the smoothness descriptor. Details about both propositions can be found in \cite{diss}.
\begin{corollary}
	Geodesic finite elements fulfill Condition~\ref{cond:3}.
\end{corollary}

\subsubsection{Condition~\ref{cond:4}: Inverse Estimates for Differences of GFEs}
In this section we want to show inverse estimates for differences of geodesic finite elements. Note that this is not the same as inverse estimates for corresponding discrete vector fields, as in general the vector field describing a difference, in our notation $\log_{v_{h}}w_{h}$ for $v_{h},w_{h}\in \gfe$, will not be in $IV(\Omega,v_{h}^{-1}TM)$.

Inverse estimates for discrete vector fields can be obtained analogously to the ones for discrete functions. Thus, one valid approach to prove Condition~\ref{cond:4} is to show that for discrete functions 
$v_{h},w_{h}\in \gfe$ a discrete Sobolev distance defined by
\begin{align*}
\left(\sum_{T_{h}\in G}\left\|\left(\log_{v_{h}}w_{h}\right)_{I}\right\|_{W^{k,p}(T_{h}, v_{h}^{-1}TM)}^{p}\right)^{\frac{1}{p}}
\end{align*}
is equivalent to the distance $d_{W^{1,p}}$ that was defined without interpolation of the vector field $\log_{v_{h}}w_{h}$. As this uses almost the same arguments as a direct proof of Condition~\ref{cond:4}, we will do the latter.

As Condition~\ref{cond:4} is about inverse estimates and thus can only be valid for discrete functions, we will in the following drop the lower index $h$ on the discrete functions for brevity of notation as there is no ambiguity.
Further we will assume that $G$ is a grid of width $h$ and order $m$, such that $F^{-1}_{h}:T\to T_{h}$ scales with order $2$ for every $T_{h}\in G$.

\begin{lemma}
	Let $u,v\in \gfe$ with $d_{L^{\infty}}(u,v)\leq \rho$ small enough, $p,q\in [1,\infty]$ with $q\leq p$. Then we have
	\begin{align*}
	d_{L^{p}}(u,v)\leq C\;h^{-d\left(\frac{1}{q}-\frac{1}{p}\right)}d_{L^{q}}(u,v)
	\end{align*}
	if $h$ is small enough such that $u(T_{h}), v(T_{h})\subset B_{\rho}$.
\end{lemma}
\begin{proof}
	We can write for any $x\in \Omega$
	\begin{align*}
	\log_{u(x)}v(x)&=\sum_{i=1}^{l}\lambda_{i}(x)\left(\log_{u(x)}v(x)-\log_{u(x)}u_{i}+d\log_{u(x)}v(x)(\log_{v(x)}v_{i}) - d\log_{u(x)}u_{i}(\log_{u_{i}}v_{i})\right)\\
	&\quad  + \sum_{i=1}^{l}\lambda_{i}(x) d\log_{u(x)}u_{i}(\log_{u_{i}}v_{i}).
	\end{align*}
	Thus, 
	\begin{align*}
	(1-C\;\rho^{2})d(u(x),v(x))\leq \Big|\sum_{i=1}^{l}\lambda_{i}(x) d\log_{u(x)}u_{i}(\log_{u_{i}}v_{i})\Big| \leq (1+C\;\rho^{2})d(u(x),v(x)).
	\end{align*}
	The class of vector fields of the form
	\begin{align*}
	\sum_{i=1}^{l}\lambda_{i}(x) d \log_{ u (x)} u_{i}(W_{i})
	\end{align*}
	is isomorphic to the finite-dimensional product space $\Pi_{i=1}^{l}T_{ u_{i}}M$. Thus, standard arguments for the norm equivalence can be applied on a reference element and scaled to the grid.
\end{proof}

\begin{lemma}\label{L:invDiff1}
	Let $u,v\in \gfe$ with $d_{L^{\infty}}(u,v)\leq \rho$ small enough. Then we have
	\begin{align*}
	D_{1,p}(u,v)\leq C\dot\theta_{1,p}(u) + h^{-1} d_{L^{p}}(u,v)
	\end{align*}
	if $h$ is small enough such that $u(T_{h}), v(T_{h})\subset B_{\rho}$.
\end{lemma}

\begin{proof}
	First derivatives of geodesic finite elements fulfill the identity
	\begin{align*}
	0&=\sum_{i=1}^{l}\lambda_{i}(x)d_{2}\log_{u(x)}u_{i}(d^{\alpha}u(x)) + \sum_{i=1}^{l}\partial_{\alpha}\lambda_{i}(x)\log_{u(x)}u_{i}
	\end{align*}
	in $T_{u(x)}M$. We set 
	\begin{align*}
	A&\colonequals d\log_{u(x)}v(x),
	\end{align*}
	and transport the corresponding identity for first derivatives of $v$ from $T_{v(x)}M$ to  $T_{u(x)}M$ using the linear map $A$.
	Note that $A$ is invertible if $\rho$ is small enough. Further, the norms of $A$ and $A^{-1}$ are bounded by a constant depending on the curvature of $M$.
	Writing
	\begin{align*}
	\nabla^{\alpha}&\log_{u(x)}v(x)
	= d\log_{u(x)}v(x)(d^{\alpha}v(x)) + 
	d_{2}\log_{u(x)}v(x)(d^{\alpha}u(x)),
	\end{align*}
	we add the identity for $v$ and substract the one for $u$ to obtain
	\begin{align*}
	\nabla^{\alpha}&\log_{u(x)}v(x)
	= \sum_{i=1}^{l}\lambda_{i}(x)A(I+d_{2}\log_{v(x)}v_{i})A^{-1}\left(\nabla^{\alpha}\log_{u(x)}v(x)\right)\\
	& - \sum_{i=1}^{l}\lambda_{i}(x)\left(A(Id+ d_{2}\log_{v(x)}v_{i})A^{-1}d_{2}\log_{u(x)}v(x) - d_{2}\log_{u(x)}u_{i}\right)(d^{\alpha} u(x))\\
	& + \sum_{i=1}^{l}\partial_{\alpha}\lambda_{i}(x)\left(d\log_{u(x)}v(x)(\log_{v(x)}v_{i}) - \log_{u(x)}u_{i} + \log_{u(x)}v(x) - d\log_{u(x)}u_{i}(\log_{u_{i}}v_{i})\right)\\
	& -  \sum_{i=1}^{l}\partial_{\alpha}\lambda_{i}(x)\log_{u(x)}v(x)
	+\sum_{i=1}^{l}\partial_{\alpha}\lambda_{i}(x) d\log_{u(x)}u_{i}(\log_{u_{i}}v_{i}).
	\end{align*}
	Thus, we can estimate
	\begin{align*}
	\left|\nabla^{\alpha}\log_{u(x)}v(x)\right|
	&\leq C\rho^{2} \left|\nabla^{\alpha}\log_{u(x)}v(x)\right|+
	C|d^{\alpha} u(x)| + C(1+\rho^{2})\;h^{-1}d(u(x),v(x)).
	\end{align*}
	For $\rho$ small enough this yields the assertion after integration and summation over the elements of the grid.
\end{proof}

\begin{lemma}
	Let $u,v\in\gfe$ with $d_{L^{\infty}}(u,v)\leq \rho$ small enough. Then we have
	\begin{multline*}
	\|\nabla^{2}\log_{u}{v}\|_{L^{p}(T_{h},u^{-1}TM)}\leq C\dot\theta_{2,p,T_{h}}(u) +C\rho\dot\theta_{1,2p,T_{h}}^{2}(v)\\
	+ C h^{-1-d\max\left\{0,\frac{1}{p}-\frac{1}{o}\right\}}D_{1,o,T_{h}}(u,v) + C\rho^{2} h^{-2} d_{L^{p}(T_{h})}(u,v)
	\end{multline*}
	if $h$ is small enough such that $u(T_{h}), v(T_{h})\subset B_{\rho}$. 
	If $m=1$, the $O(h^{-2})$-term does not appear.
\end{lemma}

\begin{proof}
	As in the proof of Lemma~\ref{L:invDiff1} we now consider identities for the second derivatives of $u,v\in \gfe$ and use the linear map $A= d\log_{u(x)}v(x)$ to transport the identity for $v$ from 
	$T_{v(x)}M$ to $T_{u(x)}M$.
	We write
	\begin{align*}
	\nabla^{\beta}\nabla^{\alpha}&\log_{u(x)}v(x)
	= d\log_{u(x)}v(x)(\nabla^{\beta}d^{\alpha}v(x))
	+	 d_{2}\log_{u(x)}v(x)(\nabla^{\beta}d^{\alpha}u(x)) \\
	&+ 
	d^{2}\log_{u(x)}v(x)(d^{\alpha}v(x),d^{\beta}v(x))
	+ d_{2}d\log_{u(x)}v(x)(d^{\alpha}v(x),d^{\beta}u(x))\\
	&
	+ dd_{2}\log_{u(x)}v(x)(d^{\alpha}u(x),d^{\beta}v(x))
	+ d_{2}^{2}\log_{u(x)}v(x)(d^{\alpha}u(x),d^{\beta}u(x)),
	\end{align*}
	and note that the leading term $A \nabla^{\beta}d^{\alpha}v(x)$ is the one we want to absorb into the right hand side. Analogously to Lemma~\ref{L:invDiff1} we add and substract the identities for second derivatives of $v$ and $u$, respectively. 
	After technical estimates we obtain for $\rho$ small enough
	\begin{align*}
	\left|\nabla^{2}\log_{u(x)}v(x)\right|
	&\leq C\;\left(\left|\nabla^{2}u(x)\right| + \rho|du|^{2}\right) + C\;\rho|dv|^{2}\\
	&\qquad
	+ C\;h^{-1} \left(|\nabla\log_{u(x)}v(x)| + \rho\; d(v(x),u(x))\right)  + C\;\rho^{2}\;h^{-2} d(u(x),v(x)),
	\end{align*}
	where the $O(h^{-2})$-term comes from second derivatives of the $\lambda_{i}$ and thus vanishes for $m=1$.
	Integration and summation over the elements of the grid then yields the assertion.
\end{proof}

\begin{corollary}
	Geodesic finite elements fulfill Condition~\ref{cond:4}.
\end{corollary}
% % % % % % % % % % % % % % % % % %
\section{Example: The Harmonic Energy}\label{sec:ch4}
In Section~\ref{sec:ch3}
we have made some fairly strong assumptions on the energy, namely that it is elliptic and predominantly quadratic.
In this section we show that these assumptions are satisfied by 
the harmonic energy $\Energy: W^{1,2}(\Omega,M)\to \RR$ defined by
\begin{equation}\label{eq:harmonicEnergy}
\Energy(v)\colonequals \frac{1}{2}\int_{\Omega}|\nabla u(x)|^{2}_{g(u(x))}\;dx = \frac{1}{2}\sum_{\alpha=1}^{d} \int_{\Omega}g_{ij}(u(x)) \frac{\partial u^{i}}{\partial x^{\alpha}}(x) \frac{\partial u^{j}}{\partial x^{\alpha}}(x) \;dx.
\end{equation}
The harmonic energy for functions into $\RR^{n}$ is the prototypic example of an elliptic quadratic energy.
For functions into a Riemannian manifold it has also served as the typical elliptic example in \cite{grohsSanderH}. The theory of stationary points of this energy, so-called harmonic maps, 
is well-developed (see, e.g., \cite{report1,report2,Jost11}). 

We will show that for smooth manifolds with bounds on the sectional curvature, the harmonic energy fulfills the assumptions of Theorem~\ref{T:L2err}.
In combination with Section~\ref{sec:ch2} this provides a priori $L^2$-discretization error estimates for geodesic finite element approximation of harmonic maps.
Corresponding numerical studies can be found in \cite{sander12} and \cite{sander13}.
Other discretization methods have been employed in \cite{bartels_prohl:2007,bartels2010numerical,alouges1997,lin1989relaxation}.

We continue to assume that $\Omega\subset \RR^{d}$ is a bounded domain with $\partial\Omega$ be in $C^{2}$.
For simplicity we assume that the given boundary and homotopy data $\phi\in C(\overline{\Omega},M)$, can be attained on the Dirichlet boundary exactly by $m$-th order geodesic finite elements on a given grid $G$ of width $h$ and order $m$.
This last restriction can be waived by suitable approximation arguments \cite{grohsSanderH}.

As the assumptions of Theorem~\ref{T:H1err} are included in the assumptions of Theorem~\ref{T:L2err}, they will also be discussed there.
\begin{lemma}\label{L:harmElliptic}
	Let $q>\max\{2,d\}$, $W^{1,q}_{K;\phi}$ be defined as in Definition~\ref{def:HKL} with boundary and homotopy data $\phi$, and let $\Energy: W^{1,2}(\Omega,M)\to \RR$ be defined by~\eqref{eq:harmonicEnergy}.
	Assume that either $M$ has nonpositive sectional curvature, or that
	\begin{align}\label{eq:curvBound}
	1-K^2\|\Rm\|_g  \Cl{c:sobolev}(q,\Omega)^{2}>0
	\end{align}
	holds, where $\Cr{c:sobolev}(q,\Omega)$ denotes the Sobolev constant for the embedding $W^{1,2}(\Omega)\subset L^{\frac{2q}{q-2}}(\Omega)$.\\
	Then $\Energy$ is elliptic in the sense of Definition~\ref{def:ellipticEnergy} along geodesic homotopies starting in any $u \in W^{1,q}_{K;\phi}$.
\end{lemma}
\begin{proof}
	The proof is based on direct calculation of the second variation (see, e.g., \cite{report1,Jost11}). Let $u\in W^{1,q}_{K;\phi}$, and 
	$V,W \in W_{0}^{1,2}(\Omega,u^{-1}TM)$.
	Then
	\begin{align*}
	\delta^{2} \Energy (u)(V,W)
	&= \int_{\Omega}\left\langle \nabla_{x} W, \nabla_{x} V \right\rangle_{g}\;dx - \int_{\Omega}\left\langle \nabla_{x} u ,R\left(\nabla_{x}u, W \right) V \right\rangle_{g}\;dx.\qedhere
	\end{align*}
	Using H\"older's inequality and the Sobolev embedding theorem, we can easily see that $\Energy$ fulfills \eqref{eq:ellipticabove} with $\Lambda= 1+K^2\|\Rm\|_g  \Cr{c:sobolev}(q,\Omega)^{2}$.
	If $M$ has nonpositive sectional curvature, then
	\begin{align*}
	\left\langle \nabla_{x} u ,R\left(\nabla_{x}u, V \right) V \right\rangle_{g}\leq 0,
	\end{align*}
	so that we obtain \eqref{eq:ellipticbelow} with $\lambda=\frac{1}{\Cr{c:poincare}^{2}}$, where $\Cl{c:poincare}$ denotes the Poincar\'e constant of the domain $\Omega$.
	Otherwise, we obtain \eqref{eq:ellipticbelow} analogously to \eqref{eq:ellipticabove} with
	\begin{align*}
	\lambda=\frac{1}{\Cr{c:poincare}^{2}}\left(1-K^2\|\Rm\|_g  \Cr{c:poincare}^{2}\Cr{c:sobolev}(q,\Omega)^{2}\right)>0.
	\end{align*}
\end{proof}
\begin{remark}
	Note that under the curvature assumptions of Lemma~\ref{L:harmElliptic}, stationary points of $\Energy$ are indeed stable critical points \cite[Corollary 8.2.2]{Jost11}.
\end{remark}
The discretization error bounds presented in Section~\ref{sec:ch3} always assume existence of solutions with a certain regularity.
The topic of existence and regularity of harmonic maps is extensively studied in the literature.
For an overview see for example \cite{report1, report2, Wood}.
In particular, we have the following.

\begin{lemma}\label{L:harmonicSmooth}
	A harmonic map $u: \Omega \to M$ with continuous boundary data $\phi$ is in $C^\infty$,
	if either $M$ has nonpositive sectional curvature, or if $d\in \{1,2\}$, or if the image of $\phi$ is contained in a geodesically convex ball.
\end{lemma}

Ellipticity and the existence of smooth solutions is enough to obtain a priori bounds for the $H^1$-discretization error (cf. Theorem~\ref{T:H1err} and \cite{grohsSanderH}).
In order to apply Theorem~\ref{T:L2err}, i.e., obtain $L^2$-error bounds as well, we need to show that $\Energy$ is also predominantly quadratic.
\begin{lemma}\label{L:almostLinearHarm}
	Let $q$, $W^{1,q}_{K;\phi}$, and  $\Energy$ be as in Lemma~\ref{L:harmElliptic}, and assume that $\Rm$ and $\nabla \Rm$ of $M$ are bounded.
	Then $\Energy$ is predominantly quadratic in the sense of Definition~\ref{def:almostLinear} at any $u\in W^{1,q}_{K;\phi}$.
\end{lemma}
\begin{proof}
	We need to consider third variations of $\Energy$ at $u\in W^{1,q}_{K;\phi}$ in directions $U\in W_{0}^{2,2}(\Omega,u^{-1}TM)$ and $V\in W_{0}^{1,2}\cap L^{d}(\Omega,u^{-1}TM)$.
	We calculate
	\begin{multline*}
	\left| \delta^{3} \Energy (u) (U,V,V) \right|
	\leq C\;\Big(\int_{\Omega} |du|^{2}\;|U|\;|V|^{2}\;dx + \int_{\Omega} |du|\;|\nabla U|\;|V|^{2}\;dx\\
	+ \int_{\Omega} |du|\;|\nabla V|\;|U|\;|V|\;dx\Big).
	\end{multline*}
	Using H\"older's inequality, the $W^{1,q}$-bound on $u$, and the Sobolev Embedding Theorem we estimate
	\begin{align*}
	|\delta^{3}\Energy (u)(U,V,V)|\leq C \|U\|_{W^{2,2}} \|V\|_{W^{1,2}}\|V\|_{L^{r}},
	\end{align*}
	where
	\begin{align*}
	\frac{1}{r}\colonequals\left\{
	\begin{array}{ll}
	\frac{1}{2}-\frac{1}{q}\;,&\ \textrm{if}\ d<4,\\
	\frac{1}{2}-\frac{1}{q}-\epsilon\;,&\ \textrm{if}\ d=4,\\
	\frac{2}{d}-\frac{1}{q}\;,&\ \textrm{if}\ d>4,
	\end{array}\right.
	\end{align*}
	with $\epsilon>0$ arbitrarily small. The condition $q>d$ implies $r\leq d$ for $d>4$.
	For $d\leq 4$, we can use the Sobolev embedding theorem to estimate $\|V\|_{L^{r}}\leq C\;\|V\|_{W^{1,2}}$. 
\end{proof}
Lastly, we need to show $H^{2}$-regularity of the solution $W\in W_{0}^{1,2}(\Omega,u^{-1}TM)$ of the deformation problem
\begin{multline}\label{eq:harmDual}
\int_{\Omega}\langle \nabla W, \nabla V \rangle \;dx -\int_{\Omega}\Rm(du,W,V,du) = \langle V,U\rangle\;dx \quad \forall V\in W_{0}^{1,2}(\Omega,u^{-1}TM),
\end{multline}
where $U\in L^{\infty}(\Omega,u^{-1}TM)$.

\begin{lemma}\label{L:dualH2regHarm}
	Let $q$ and $W^{1,q}_{K;\phi}$ be as in Lemma~\ref{L:harmElliptic}, and $b$ as in Definition~\ref{def:nonlinSmoothVfield}.
	Assume that $\Rm$ of $M$ is bounded, and let $u\in W_{K;\phi}^{1,q}\cap W^{2,b}(\Omega,M)$ be a harmonic map.
	Then the deformation problem \eqref{eq:harmDual} is $H^{2}$-regular in the sense of \eqref{eq:H2reg}.
\end{lemma}
\begin{proof}
	The deformation problem \eqref{eq:harmDual} is essentially a linear system of elliptic equations in divergence form. One can use coordinates to write the covariant derivatives in terms of ordinary derivatives and Christoffel symbols to see this. As $u\in W_{K}^{1,q}$, the coefficients are smooth enough to apply standard theory for linear systems of elliptic equations (see, e.g., \cite{lady}) to obtain an estimate of the form
	\begin{align*}
	\|\hat{W}\|_{W^{2,2}(\Omega,\RR^{n})}\leq C\|U\|_{L^{2}},
	\end{align*}
	where $\hat{W}$ is the coordinate vector for $W$. As $u\in W^{2,b}(\Omega, M)$, we can estimate the covariant norm by the coordinate one and obtain
	\begin{align*}
	\|W\|_{W^{2,2}(\Omega,u^{-1}TM)}\leq  C\|U\|_{L^{2}}.
	\end{align*}
	An alternative proof directly shows $H^{2}$-regularity as the standard theory does, using parallel transport to define difference quotients for vector fields $V$ along $u$.
	Details can be found in the appendix of \cite{diss}.
\end{proof}
Summing up, the harmonic energy fulfills all assumptions of Theorem~\ref{T:L2err}, and we obtain the following discretization error bound.
\begin{theorem}\label{T:L2harm}
	Let $m\geq 1$ with $2(m+1)>d$.
	Assume that $G$ is a conforming grid of width $h$ and order $m$, such that $F^{-1}_{h}:T\to T_{h}$ scales with order $2$.
	
	Let $u$ be a local minimizer of the harmonic energy $\Energy$ on $W^{1,2}_{\phi}(\Omega,M)$, 
	where $M$ has either nonpositive sectional curvature, or \eqref{eq:curvBound} holds.
	If $M$ does not have nonpositive sectional curvature, we additionally assume that either $d \in \{1,2\}$, or that $\phi(\Omega)$ is contained in a geodesically convex ball.
	
	If $h$ is small enough, there exists a local minimizer $u_{h}$ in the set of geodesic finite elements $\gfe$ such that
	\begin{align}\label{eq:L2harmIntr}
	d_{L^{2}}(u,u_{h})\leq C\;h^{m+1}\theta^{2}_{m+1,2,\Omega}(u).
	\end{align}
\end{theorem}

\begin{remark}
	Numerical experiments in \cite{sander13, sander12} confirm the result of Theorem~\ref{T:L2harm} for a test case with $M=S^{2}$.
\end{remark}

\bibliographystyle{abbrv}     
\bibliography{literatur}

\begin{thebibliography}{10}

\bibitem{alouges1997}
F.~Alouges.
\newblock A new algorithm for computing liquid crystal stable configurations:
  the harmonic mapping case.
\newblock {\em SIAM J. Numer. Anal.}, 34(5):1708--1726, 1997.

\bibitem{Ambrosio}
L.~Ambrosio, N.~Gigli, and G.~Savar{\'e}.
\newblock {\em Gradient flows in metric spaces and in the space of probability
  measures}.
\newblock Springer, 2006.

\bibitem{arnold2006finite}
D.~N. Arnold, R.~S. Falk, and R.~Winther.
\newblock Finite element exterior calculus, homological techniques, and
  applications.
\newblock {\em Acta numerica}, 15:1--155, 2006.

\bibitem{bartels2005}
S.~Bartels.
\newblock Stability and convergence of finite-element approximation schemes for
  harmonic maps.
\newblock {\em SIAM J. Numer. Anal.}, 43(1):220--238, 2005.

\bibitem{bartels2010numerical}
S.~Bartels.
\newblock Numerical analysis of a finite element scheme for the approximation
  of harmonic maps into surfaces.
\newblock {\em Math. Comput.}, 79(271):1263--1301, 2010.

\bibitem{bartels_prohl:2007}
S.~Bartels and A.~Prohl.
\newblock Constraint preserving implicit finite element discretization of
  harmonic map flow into spheres.
\newblock {\em Math. Comput.}, 76(260):1847--1859, 2007.

\bibitem{Casciaro}
B.~Casciaro and M.~Francaviglia.
\newblock A new variational characterization of {J}acobi fields along
  geodesics.
\newblock {\em Ann. Mat. Pur. Appl.}, 172(4):219--228, 1997.

\bibitem{ciarlet}
P.~G. Ciarlet.
\newblock {\em The finite element method for elliptic problems}.
\newblock Elsevier, 1978.

\bibitem{deckelnick2005computation}
K.~Deckelnick, G.~Dziuk, and C.~M. Elliott.
\newblock Computation of geometric partial differential equations and mean
  curvature flow.
\newblock {\em Acta Numerica}, 14:139--232, 2005.

\bibitem{Rannacher}
M.~Dobrowolski and R.~Rannacher.
\newblock Finite element methods for nonlinear elliptic systems of second
  order.
\newblock {\em Math. Nachr.}, 94:155--172, 1980.

\bibitem{dziuk2013L2}
G.~Dziuk and C.~Elliott.
\newblock ${L}^2$-estimates for the evolving surface finite element method.
\newblock {\em Math. Comput.}, 82(281):1--24, 2013.

\bibitem{dziuk2007finite}
G.~Dziuk and C.~M. Elliott.
\newblock Finite elements on evolving surfaces.
\newblock {\em IMA J. Numer. Anal.}, 27(2):262--292, 2007.

\bibitem{dziuk2012fully}
G.~Dziuk and C.~M. Elliott.
\newblock A fully discrete evolving surface finite element method.
\newblock {\em SIAM J. Numer. Anal.}, 50(5):2677--2694, 2012.

\bibitem{report1}
J.~Eels and L.~Lemaire.
\newblock A report on harmonic maps.
\newblock {\em B. Lond. Math. Soc.}, 10(1):1--68, 1978.

\bibitem{report2}
J.~Eels and L.~Lemaire.
\newblock Another report on harmonic maps.
\newblock {\em B. Lond. Math. Soc.}, 20(1):385--524, 1988.

\bibitem{eichmair2013}
M.~Eichmair and J.~Metzger.
\newblock Large isopserimetric surfaces in initial data sets.
\newblock {\em J. Differ. Geom.}, 94(1):159--186, 2013.

\bibitem{Evans}
L.~C. Evans.
\newblock {\em Partial Differential Equations}.
\newblock AMS, 1998.

\bibitem{grohs}
P.~Grohs.
\newblock Quasi-interpolation in {R}iemannian manifolds.
\newblock {\em IMA J. Numer. Anal.}, 33(3):849--874, 2013.

\bibitem{grohsSanderH}
P.~Grohs, H.~Hardering, and O.~Sander.
\newblock Optimal a priori discretization error bounds for geodesic finite
  elements.
\newblock {\em Found. Comput. Math.}, 15(6):1357--1411, 2015.

\bibitem{Hajlasz}
P.~Haj{\l}asz.
\newblock {S}obolev mappings between manifolds and metric spaces.
\newblock In {\em {S}obolev Spaces In Mathematics I}, volume~8 of {\em
  International Mathematical Series}, pages 185--222. Springer, 2009.

\bibitem{diss}
H.~Hardering.
\newblock {\em Intrinsic Discretization Error Bounds for Geodesic Finite
  Elements}.
\newblock PhD thesis, Freie Universit{\"a}t Berlin, 2015.

\bibitem{Helein}
F.~H{\'{e}}lein.
\newblock {\em Harmonic Maps, Conservation Laws and Moving Frames}.
\newblock Cambridge University Press, second edition, 2002.

\bibitem{Wood}
F.~H{\'{e}}lein and J.~C. Wood.
\newblock Harmonic maps.
\newblock In {\em Handbook of global analysis}, pages 417--491. Elsevier, 2008.

\bibitem{Jost11}
J.~Jost.
\newblock {\em {R}iemannian geometry and geometric analysis}.
\newblock Springer, 2008.

\bibitem{Karcher1977}
H.~Karcher.
\newblock Mollifier smoothing and {R}iemannian center of mass.
\newblock {\em Commun. Pur. Appl. Math.}, 30:509--541, 1977.

\bibitem{Kowalski}
O.~Kowalski and M.~Sekizawa.
\newblock Natural transformations of {R}iemannian metrics on manifolds to
  metrics on tangent bundles -- a classification.
\newblock {\em Bull. Tokyo Gakugei Univ.}, 40:1--29, 1997.

\bibitem{lady}
O.~A. Ladyzhenskaya and N.~N. Ural'tseva.
\newblock {\em Linear and Quasilinear Elliptic Equations}.
\newblock AP, 1968.

\bibitem{lin1989relaxation}
S.-Y. Lin and M.~Luskin.
\newblock Relaxation methods for liquid crystal problems.
\newblock {\em SIAM J. Numer. Anal.}, 26(6):1310--1324, 1989.

\bibitem{Melcher}
C.~Melcher.
\newblock Chiral skyrmions in the plane.
\newblock {\em P. Roy. Soc. Lond. A: Mat.}, 470(2172), 2014.

\bibitem{muench:2007}
I.~M{\"{u}}nch.
\newblock {\em Ein geometrisch und materiell nichtlineares {C}osserat-Model ---
  {T}heorie, {N}umerik und {A}nwendungsm\"{o}glichkeiten}.
\newblock PhD thesis, Universit\"{a}t Karlsruhe, 2007.

\bibitem{muench_neff_wagner:2011}
I.~M{\"{u}}nch, P.~Neff, and W.~Wagner.
\newblock Transversely isotropic material: nonlinear {C}osserat versus
  classical approach.
\newblock {\em Continuum Mech. Therm.}, 23(1):27--34, 2011.

\bibitem{muench_wagner_neff:2009}
I.~M{\"{u}}nch, W.~Wagner, and P.~Neff.
\newblock Theory and {FE}-analysis for structures with large deformation under
  magnetic loading.
\newblock {\em Comput. Mech.}, 44(1):93--102, 2009.

\bibitem{Nash}
J.~Nash.
\newblock The imbedding problem for {R}iemannian manifolds.
\newblock {\em Ann. Math.}, 63(1):20--63, 1956.

\bibitem{neff2007geometrically}
P.~Neff.
\newblock A geometrically exact planar cosserat shell-model with
  microstructure: Existence of minimizers for zero cosserat couple modulus.
\newblock {\em Math. Mod. Meth. Appl. S.}, 17(03):363--392, 2007.

\bibitem{Palais}
R.~S. Palais.
\newblock {\em Foundations of global non-linear analysis}.
\newblock Benjamin, 1968.

\bibitem{Riviere}
T.~Rivi{\`e}re.
\newblock Everywhere discontinuous harmonic maps into spheres.
\newblock {\em Acta Mathematica}, 175(2):197--226, 1995.

\bibitem{sander10}
O.~Sander.
\newblock Geodesic finite elements for {C}osserat rods.
\newblock {\em Int. J. Num. Meth. Eng.}, 82(13):1645--1670, 2010.

\bibitem{sander12}
O.~Sander.
\newblock Geodesic finite elements on simplicial grids.
\newblock {\em Int. J. Num. Meth. Eng.}, 92(12):999--1025, 2012.

\bibitem{sander13}
O.~Sander.
\newblock Geodesic finite elements of higher order.
\newblock {\em IMA J. Numer. Anal.}, 36(1):238--266, 2015.

\bibitem{SanderTest}
O.~Sander.
\newblock Test function spaces for geometric finite elements.
\newblock arXiv: 1607.07479, 2016.

\bibitem{SanderNeff}
O.~Sander, P.~Neff, and M.~B{\^{i}}rsan.
\newblock Numerical treatment of a geometrically nonlinear planar {C}osserat
  shell model.
\newblock {\em Comp. Mech.}, 57(5):817--841, 2016.

\bibitem{schmeisser2005vector}
H.-J. Schmei{\ss}er and W.~Sickel.
\newblock Vector-valued {S}obolev spaces and {G}agliardo-{N}irenberg
  inequalities.
\newblock In {\em Nonlinear elliptic and parabolic problems}, pages 463--472.
  Birkh{\"a}user, 2005.

\bibitem{Uhlenbeck}
R.~Schoen and K.~Uhlenbeck.
\newblock Boundary regularity and the {D}irichlet problem for harmonic maps.
\newblock {\em J. Differ. Geom.}, 18:253--268, 1983.

\bibitem{simo_fox_rifai:1990}
J.~C. Simo, D.~D. Fox, and M.~S. Rifai.
\newblock On a stress resultant geometrically exact shell model. {P}art {III}:
  Computational aspects of the nonlinear theory.
\newblock {\em Comput. Methods Appl. Mech. Engrg.}, 79(1):21--70, 1990.

\bibitem{simo_vu-quoc:1986}
J.~C. Simo and L.~Vu-Quoc.
\newblock A three-dimensional finite-strain rod model. {P}art~{II}:
  Computational aspects.
\newblock {\em Comput. Methods Appl. Mech. Engrg.}, 58(1):79--116, 1986.

\bibitem{struwe1985}
M.~Struwe.
\newblock On the evolution of harmonic mappings of {R}iemannian surfaces.
\newblock {\em Comment. Math. Helv.}, 60(1):558--581, 1985.

\bibitem{vese2002numerical}
L.~A. Vese and S.~J. Osher.
\newblock Numerical methods for p-harmonic flows and applications to image
  processing.
\newblock {\em SIAM J. Numer. Anal.}, 40(6):2085--2104, 2002.

\bibitem{wriggers_gruttmann:1993}
P.~Wriggers and F.~Gruttmann.
\newblock Thin shells with finite rotations formulated in {B}iot stresses:
  Theory and finite element formulation.
\newblock {\em Int. J. Num. Meth. Eng.}, 36:2049--2071, 1993.

\end{thebibliography}

%\listoftodos

\end{document}